\documentclass[american,3p,lefttitle]{elsarticle}
\pdfoutput=1
\journal{Journal of Computational and Applied Mathematics}
\biboptions{longnamesfirst, sort&compress}

\usepackage[utf8]{inputenc}
\usepackage[T1]{fontenc}
\usepackage{babel, lmodern}
\usepackage[babel=true]{microtype}
\usepackage{amsmath, amssymb, amsthm, booktabs, mathtools, pgfplots, subcaption}
\usepackage[export]{adjustbox}

\newcommand{\Author}{W. Huang, L. Kamenski, and J. Lang}
\newcommand{\Title}{%
   Conditioning of~implicit Runge--Kutta integration for~finite element
   approximation of~linear diffusion equations on~anisotropic meshes}

\usepackage{hyperref}
\hypersetup{pdfauthor={\Author{}}, pdftitle={\Title{}}, pdfsubject={Revised Manuscript}}

\usepackage[capitalize]{cleveref}
\crefformat{section}{#2Sect.~#1#3}
\Crefformat{section}{#2Section~#1#3}
\crefmultiformat{section}{Sects.~#2#1#3}%
   { and~#2#1#3}{, #2#1#3}{ and~#2#1#3}
\crefformat{equation}{(#2#1#3)}
\Crefformat{equation}{Equation~(#2#1#3)}
\crefmultiformat{equation}{(#2#1#3)}%
   { and~(#2#1#3)}{, (#2#1#3)}{ and~(#2#1#3)}
\crefrangeformat{equation}{(#3#1#4) to~(#5#2#6)}

\usetikzlibrary{shadings}
\pgfplotsset{%
   compat=newest, grid=major, ticks=major, legend cell align=left,
   every axis/.append style={%
      width=0.53\textwidth, height=0.304\textheight,
      font=\small, legend columns=2,
      legend style={row sep=-0.5ex, at = {(0,1)}, anchor = north west},
      xtick={1.0e+02, 1.0e+03, 1.0e+04, 1.0e+05, 1.0e+06},
      ytick={1.0e+00, 1.0e+01, 1.0e+02, 1.0e+03, 1.0e+04, 1.0e+05, 1.0e+06, 1.0e+07},
   },
}
\newcommand{\MyLineWidth}{1pt}
\newcommand{\MyMarkSize}{2pt}

\newcommand{\RKmethod}{Radau~IA}
\newcommand{\RKshort}{RIA3}

\newcommand{\plotRadauTEuler}[9]{%
   \centering{}%
   \begin{tikzpicture}%
      \begin{loglogaxis}[
         ymin=1e-00, ymax=4.0e+08,
         #9 ]
         \addplot[color=red, dashed, line width=\MyLineWidth,
            mark size=\MyMarkSize, mark=#7, mark options={solid} ]
            table [x index=2, y index=#3, col sep = space]
               {./#1-e1-#2.dat};
         \addlegendentry{$10^{-1}$#5}
         \addplot[color=red, dashed, line width=\MyLineWidth,
            mark size=\MyMarkSize, mark=#8, mark options={solid} ]
            table [x index=2, y index=#4, col sep = space]
            {./#1-e1-#2.dat};
         \addlegendentry{$10^{-1}$#6}
         \addplot[color=blue, dotted, line width=\MyLineWidth,
            mark size=\MyMarkSize, mark=#7, mark options={solid} ]
            table [x index=2, y index=#3, col sep = space]
               {./#1-e3-#2.dat};
         \addlegendentry{$10^{-3}$#5}
         \addplot[color=blue, dotted, line width=\MyLineWidth,
            mark size=\MyMarkSize, mark=#8, mark options={solid} ]
            table [x index=2, y index=#4, col sep = space]
            {./#1-e3-#2.dat};
         \addlegendentry{$10^{-3}$#6}
         \addplot[color=black, solid, line width=\MyLineWidth,
            mark size=\MyMarkSize, mark=#7, mark options={solid} ]
            table [x index=2, y index=#3, col sep = space]
            {./#1-dx-#2.dat};
         \addlegendentry{$N^{-\frac{1}{2}}$#5}
         \addplot[color=black, solid, line width=\MyLineWidth,
            mark size=\MyMarkSize, mark=#8, mark options={solid} ]
            table [x index=2, y index=#4, col sep = space]
            {./#1-dx-#2.dat};
         \addlegendentry{$N^{-\frac{1}{2}}$#6}
         \addplot[patch, table/row sep=\\,
            colormap={bw}{gray(0cm)=(0); gray(1cm)=(0)}, fill=white,]
            table {1.0e+05 1.0e+02\\ 1.0e+06 1.0e+02\\ 1.0e+06 1.0e+03\\};
         \addplot[patch, table/row sep=\\,
            colormap={bw}{gray(0cm)=(0); gray(1cm)=(0)}, fill=white,]
            table {1.0e+05 1.0e+02\\ 1.0e+06 1.0e+02\\ 1.0e+06 3.3e+02\\};
      \end{loglogaxis}%
   \end{tikzpicture}%
}

\newcommand{\plotScalingE}[9]{%
   \centering{}%
   \begin{tikzpicture}
      \begin{loglogaxis}[
         ymin=1e+00, ymax=4.0e+10,
         ytick={1.0e+01, 1.0e+02, 1.0e+03, 1.0e+04, 1.0e+05, 1.0e+06, 1.0e+07, 1.0e+08},
         #9 ]
         \addplot[color=red, dashed, line width=\MyLineWidth,
            mark size=\MyMarkSize, mark=#7, mark options={solid} ]
            table [x index=2, y index=#4, col sep = space]
            {./#1-e1-#3.dat};
         \addlegendentry{$10^{-1}$#5}
         \addplot[color=red, dashed, line width=\MyLineWidth,
            mark size=\MyMarkSize, mark=#8, mark options={solid} ]
            table [x index=2, y index=#4, col sep = space]
            {./#2-e1-#3.dat};
         \addlegendentry{$10^{-1}$#6}
         \addplot[color=blue, dotted, line width=\MyLineWidth,
            mark size=\MyMarkSize, mark=#7, mark options={solid} ]
            table [x index=2, y index=#4, col sep = space]
            {./#1-e3-#3.dat};
         \addlegendentry{$10^{-3}$#5}
         \addplot[color=blue, dotted, line width=\MyLineWidth,
            mark size=\MyMarkSize, mark=#8, mark options={solid} ]
            table [x index=2, y index=#4, col sep = space]
            {./#2-e3-#3.dat};
         \addlegendentry{$10^{-3}$#6}
         \addplot[color=black, solid, line width=\MyLineWidth,
            mark size=\MyMarkSize, mark=#7, mark options={solid} ]
            table [x index=2, y index=#4, col sep = space]
            {./#1-e5-#3.dat};
         \addlegendentry{$10^{-5}$#5}
         \addplot[color=black, solid, line width=\MyLineWidth,
            mark size=\MyMarkSize, mark=#8, mark options={solid} ]
            table [x index=2, y index=#4, col sep = space]
            {./#2-e5-#3.dat};
         \addlegendentry{$10^{-5}$#6}
         \addplot[patch, table/row sep=\\,
            colormap={bw}{gray(0cm)=(0); gray(1cm)=(0)}, fill=white,]
            table {1.0e+05 1.0e+02\\ 1.0e+06 1.0e+02\\ 1.0e+06 1.0e+03\\};
      \end{loglogaxis}
   \end{tikzpicture}%
}

\newcommand{\A}{\mathbb{A}}
\newcommand{\D}{\mathbb{D}}
\newcommand{\Dt}{\Delta{} t}
\newcommand\M{\mathbb{M}}
\newcommand\R{\mathbb{R}}
\newcommand\Th{\mathcal{T}_h}
\newcommand{\hD}{h_{\D^{-1}}}
\newcommand{\hKD}{h_{K,\D^{-1}}}
\newcommand{\aKD}{a_{K,\D^{-1}}}
\newcommand{\bu}{\boldsymbol{u}}
\newcommand{\bv}{\boldsymbol{v}}
\newcommand{\bw}{\boldsymbol{w}}
\newcommand{\bx}{\boldsymbol{x}}
\newcommand{\dx}{\;\mathrm{d}\boldsymbol{x}}
\newcommand{\FDF}{{(F_K')}^{-1} \D_K {(F_K')}^{-T}}
\newcommand{\Abs}[1]{{\left\lvert#1\right\rvert}}
\newcommand{\abs}[1]{{\lvert#1\rvert}}

\newcommand{\norm}[1]{{\lVert#1\rVert}}

\newcommand{\EnonU}{\Psi_{E}}
\newcommand{\DnonU}{\Psi_{\D}}
\newcommand{\nonUinLog}{\psi_{\D}}

\newtheorem{lemma}{Lemma}[section]
\newtheorem{theorem}{Theorem}[section]

\newdefinition{example}{Example}[section]
\newdefinition{remark}{Remark}[section]

\begin{document}

\date{October 5, 2019}
\begin{frontmatter}

\title{\Title{}\tnoteref{t1}}
\tnotetext[t1]{%
      Supported in part by
      the DFG under grant KA\,3215/2-1
      and the Darmstadt Graduate Schools of Excellence
         \emph{Computational Engineering}
         and~\emph{Energy Science and Engineering}.%
}

\author[1]{Weizhang Huang}
\ead{whuang@ku.edu}
\address[1]{Department of Mathematics, The University of Kansas,
      Lawrence, KS~66045, USA}

\author[2]{Lennard Kamenski\corref{cor1}}
\ead{lkamenski@m4sim.de}
\address[2]{m4sim GmbH, 10117~Berlin, Germany}

\author[3]{Jens Lang}
\ead{lang@mathematik.tu-darmstadt.de}
\address[3]{Department of Mathematics, TU Darmstadt,
   64293~Darmstadt, Germany}

   \cortext[cor1]{Corresponding author.}

\begin{abstract}
The conditioning of implicit Runge--Kutta (RK) integration for linear finite element approximation of diffusion equations on general anisotropic meshes is investigated.
Bounds are established for the condition number of the resulting linear system with and without diagonal preconditioning for the implicit Euler (the simplest implicit RK method) and general implicit RK methods.
Two solution strategies are considered for the linear system resulting from general implicit RK integration: the simultaneous solution where the system is solved as a whole and a successive solution which follows the commonly used implementation of implicit RK methods to first transform the system into a number of smaller systems using the Jordan normal form of the RK matrix and then solve them successively.

For the simultaneous solution in case of a positive semidefinite symmetric part of the RK coefficient matrix and for the successive solution it is shown that the conditioning of an implicit RK method behaves like that of the implicit Euler method.
If the smallest eigenvalue of the symmetric part of the RK coefficient matrix is negative and the simultaneous solution strategy is used, an upper bound on the time step is given so that the system matrix is positive definite.

The obtained bounds for the condition number have explicit geometric interpretations and take the interplay between the diffusion matrix and the mesh geometry into full consideration.
They show that there are three mesh-dependent factors that can affect the conditioning: the number of elements, the mesh nonuniformity measured in the Euclidean metric, and the mesh nonuniformity with respect to the inverse of the diffusion matrix.
They also reveal that the preconditioning using the diagonal of the system matrix, the mass matrix, or the lumped mass matrix can effectively eliminate the effects of the mesh nonuniformity measured in the Euclidean metric.
Illustrative numerical examples are given.
\end{abstract}

\begin{keyword}
   finite element method, anisotropic mesh, condition number,
   parabolic equation, implicit Runge--Kutta method
\MSC[2010]
   65M60, 
   65M50, 
   65F35, 
   65F15  
\\
\par\noindent\rule{\textwidth}{0.45pt}\\
\vspace{0.3\baselineskip}
{\small{This is a preprint of a contibution published by Elsevier in
\href{https://doi.org/10.1016/j.cam.2019.112497}{\emph{J.\ Comput.\ Appl.\ Math.} (2019) 112497}}.}%
\\[0.4\baselineskip]%
{\footnotesize{}%
   \begin{tabular}{@{}p{0.14\linewidth}@{}p{0.85\linewidth}@{}}
      \includegraphics[height=2.2\baselineskip, valign=t]{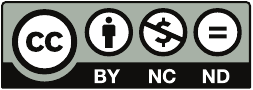}
   &
   \copyright~2019. Licensed under CC-BY-NC-ND~4.0 (\url{https://creativecommons.org/licenses/by-nc-nd/4.0}).
   \vspace{0.3\baselineskip}
   \newline{}
   The formal publication is available online at \url{https://doi.org/10.1016/j.cam.2019.112497}.
   \end{tabular}%
}%
\end{keyword}

\end{frontmatter}

\section{Introduction}\label{sec:introduction}

The nonuniformity of adaptive meshes has considerable effects on the conditioning of the discrete approximation of partial differential equations (PDEs) and their efficient solution.
To study these effects, we investigate the implicit Runge--Kutta (RK) integration for
the linear finite element (FE) approximation of linear diffusion equations on general
simplicial anisotropic meshes for the initial-boundary value problem (IBVP)
\begin{align}
   \begin{cases}
      \begin{alignedat}{3}
         & \partial_t u = \nabla \cdot \left(\D \nabla u \right),
            \qquad && \bx \in \Omega,
            \quad  && t > 0,
            \\
         & u(\bx, t) = 0,
            \qquad && \bx \in \partial \Omega,
            \quad  && t > 0,
            \\
         & u(\bx, 0) = u_0(\bx),
            \qquad && \bx \in \Omega,
      \end{alignedat}
   \end{cases}
   \label{eq:IBVP}
\end{align}
where $\Omega \subset \mathbb{R}^d$ ($d \ge 1$) is a bounded polygonal or polyhedral domain, $u_0$ is a given initial solution, and $\D$ is the diffusion matrix.
We assume that $\D = \D(\bx)$ is time independent and symmetric and uniformly positive definite on $\Omega$, i.e.,
\begin{equation}
   \exists \, d_{\min}, d_{\max} > 0 \colon\quad
      d_{\min} I \le \D(\bx) \le d_{\max} I,
      \quad \forall \bx \in \Omega,
   \label{eq:D:i}
\end{equation}
where the less-than-or-equal sign means that the difference between the two matrices is positive semidefinite.
We consider the Dirichlet boundary condition in this work
but the analysis is applicable to other boundary condition types without major modifications.

Much effort has been made in the past to understand the effects of mesh nonuniformity
on the conditioning of FE approximations.
For example, Fried~\cite{Fri73} obtains a bound on the condition number of the stiffness matrix for the linear FE approximation of the Laplace operator for general meshes.
For the Laplace operator on isotropic adaptive grids, Bank and Scott~\cite{BanSco89} show that the condition number of the diagonally scaled stiffness matrix is essentially the same as for a regular mesh.
Ainsworth, McLean, and Tran~\cite{AinMcLTra99} and Graham and McLean~\cite{GraMcL06} extend this result to the boundary element equations for locally quasi-uniform meshes and provide a bound in terms of patch volumes and aspect ratios.
Du et al.~\cite{DuWanZhu09} obtain a bound on the condition number of the stiffness matrix for a general diffusion operator on anisotropic meshes which reveals the relation between the condition number and some mesh quality measures.
For the FE approximations of parabolic problems, Zhu and Du~\cite{ZhuDu11,ZhuDu14} develop mesh dependent stability and condition number estimates for the explicit and implicit Euler methods.
In the case of a lumped mass matrix ($M_{lump}$), Shewchuk~\cite{She02} provides a bound
on the largest eigenvalue of $M_{lump}^{-1} A$ in terms of the maximum eigenvalues of local element matrices,
where $A$ is the stiffness matrix.
The results mentioned above allow anisotropic adaptive meshes but do not fully take into account the interplay between the mesh geometry and the diffusion matrix. (An exception is the bound by Shewchuk
which includes the effects of the diffusion coefficients.)
Moreover, the existing analysis employs mesh restrictions in form of mesh regularity assumptions, e.g.,\ local mesh uniformity, or parameters in final estimates that are related to mesh regularity, such as the maximum ratio of volumes of neighbouring elements or the maximum number of elements in a patch.

The objective of this work is to develop estimates on the conditioning of the resulting linear system that take the interplay between the mesh geometry and the diffusion matrix into full consideration, have explicit geometric interpretation, and make no prior assumptions on the mesh regularity.
This is a continuation of our previous effort to develop bounds for the condition number of the stiffness matrix for the linear FE equations of a general diffusion operator on arbitrary anisotropic meshes~\cite{KamHua14a,KamHuaXu14} and the largest permissible time steps for explicit RK schemes for both linear and high order FE approximations of the IBVP~\cref{eq:IBVP}~\cite{HuaKamLan15,HuaKamLan16}.
In particular, these bounds show~\cite{KamHuaXu14} that the condition number of the stiffness matrix depends on three factors: the factor depending on the number of mesh elements and corresponding to the condition number of the linear FE equations for the Laplace operator on a uniform mesh, the nonuniformity of the mesh viewed in the metric defined by the inverse diffusion matrix, $\D^{-1}$, and the mesh nonuniformity measured in the Euclidean metric.
Moreover, the Jacobi preconditioning, an optimal diagonal scaling for a symmetric positive definite sparse matrix~\cite[Corollary~7.6ff.]{Hig96}, can effectively eliminate the effects of mesh nonuniformity and reduce those of the mesh nonuniformity with respect to $\D^{-1}$~\cite{KamHuaXu14}.
Detailed characterizations of the condition number according to the mesh concentration distribution can be obtained using Green's functions~\cite{Fri73,KamHua14a}.

We first consider the implicit Euler method (the simplest implicit RK method) and establish bounds for the condition number of the corresponding system matrix with and without diagonal scaling.

For general implicit RK methods, we consider two strategies for solving the resulting system: the simultaneous solution and a successive solution.
We need to point out that for general implicit RK methods, the coefficient matrix is not necessarily symmetric
and thus a condition number in the standard definition will not provide much information for iterative solvers.
Motivated by the convergence estimates by Eisenstat et al.~\cite{EisElmSch83}
for the generalized minimal residual method (GMRES) for nonsymmetric systems,
we use here the ratio of the maximum singular value of the coefficient matrix and the minimum eigenvalue
of its symmetric part to measure the conditioning.

For the simultaneous solution, the system is solved as a whole.
Obtained estimates reveal that the conditioning of implicit RK methods with positive semidefinite symmetric part of the coefficient matrix behaves like that of the implicit Euler method (\cref{thm:irk:1} in \cref{sec:irk}).
If the smallest eigenvalue of the symmetric part of the RK coefficient matrix is negative, we provide an upper bound \cref{dt-0} on the possible time step so that the system matrix is positive definite.
This condition can be serious and lead to $\Delta t = \mathcal{O}(h^2)$.%

For the successive solution, which follows the commonly used implementation of implicit RK methods~\cite{Bic77,But76},
the system is first transformed into a number of smaller systems using the Jordan normal form of the RK matrix and then solved successively.
We show that the conditioning of the implicit RK integration behaves like that of the implicit Euler method and is determined by the conditioning of two types of matrices.
The first one is similar to the implicit Euler method and corresponds to the real eigenvalues of the RK matrix, while the second, twice as large, corresponds to the complex eigenvalues of the RK matrix (cf.~\cref{eq:irk-2-1} in~\cref{sec:irk}).

The paper is organized as follows.
We first introduce the FE formulation and its implicit RK integration (\cref{sec:setting}) and provide preliminary estimates for the extremal eigenvalues of the mass and stiffness matrices (\cref{sec:preliminary}).
The main results for the conditioning of the coefficient matrices are given in \cref{sec:euler,sec:irk}, followed by numerical examples (\cref{sec:numerics}) and conclusions (\cref{sec:conclusion}).

\section{Linear FE approximation and implicit RK integration}\label{sec:setting}

Let $\{\Th\}$ be a given family of simplicial meshes for the domain $\Omega$ and $N$, $N_v$, and $N_{vi}$ the number of mesh elements, vertices, and interior vertices, respectively.
For convenience, we assume that the vertices are ordered such that the first $N_{vi}$ vertices are the interior ones.
The element patch associated with the $j$-th vertex is denoted by $\omega_j$, $K$ denotes a given mesh element and $\hat{K}$ is the reference element which is assumed to have been taken as a unitary equilateral simplex.
Element and patch volumes are denoted by $\Abs{K}$ and $\Abs{\omega_j} = \sum_{K \in \omega_j} \Abs{K}$.
For each mesh element $K\in \Th$, we denote the invertible affine mapping from $\hat{K}$ to $K$ and its Jacobian matrix by $F_K$ and $F_K'$, respectively.
Note that $F_K'$ is a constant matrix and $\det(F_K') = \Abs{K}$.

Let $V^h \subset H_0^1(\Omega)$ be the linear FE space associated with $\Th$.
The piecewise linear FE solution $u_h(t) \in V^h$, $t > 0$ for \cref{eq:IBVP} is defined by
\begin{equation}
   \int_\Omega \frac{\partial u_h}{\partial t} v_h \dx
   = - \int_\Omega {(\nabla v_h)}^T \D \nabla u_h \dx ,
      \qquad \forall v_h \in V^h, \quad t > 0,
         \label{eq:FEM:i}
\end{equation}
subject to the initial condition
\[
   \int_\Omega u_h(\bx, 0) v_h \dx
      = \int_\Omega u^0(\bx) v_h \dx ,
      \qquad \forall v_h \in V^h .
\]
It can be expressed as
\[
   u_h(\bx, t) = \sum_{j=1}^{N_{vi}} u_j(t) \phi_j (\bx),
\]
where $\phi_j $ is the linear basis function associated with the $j$-th vertex.
Inserting this into \cref{eq:FEM:i} and taking $v_h = \phi_k$, $k = 1, \dotsc, N_{vi}$, successively yields the system
\begin{equation}
   M \frac{d \bu}{d t} = - A \bu,
   \label{eq:fem}
\end{equation}
where $\bu = {(u_1,\dotsc, u_{N_{vi}})}^T$ and $M$ and $A$ are the mass and stiffness matrices with
\begin{equation}
   M_{kj} = \int_\Omega \phi_k \phi_j \dx
   \quad \text{and} \quad
   A_{kj} = \int_\Omega \nabla \phi_k \cdot \D \nabla \phi_j \dx ,
   \qquad k, j = 1, \dotsc, N_{vi}.
   \label{eq:AM:i}
\end{equation}

For the time integration of \cref{eq:fem} we consider a general implicit $s$-stage RK method with the Butcher tableau
\begin{center}%
\begin{tabular}{c|cccc}
   $c_1$    & $\gamma_{11}$  & $\gamma_{12}$  & $\cdots$  & $\gamma_{1s}$\\
   $c_2$    & $\gamma_{21}$  & $\gamma_{22}$  & $\cdots$  & $\gamma_{2s}$\\
   $\vdots$ & $\vdots$       & $\vdots$       &           & $\vdots$\\
   $c_s$    & $\gamma_{s1}$  & $\gamma_{s2}$  & $\cdots$  & $\gamma_{ss}$\\
   \hline
            & $b_1$          & $b_2$          & $\cdots$  & $b_s$
\end{tabular}%
\end{center}
and assume that the eigenvalues of the RK matrix have nonnegative real parts.
This requirement is satisfied by most implicit RK methods (e.g.,\ see \cref{tab:t1}).
Applying the method to \cref{eq:fem} yields
\begin{align}
   & M \bv_k + \Dt A \sum_{j=1}^s \gamma_{kj} \bv_j
   = - A \bu^n,
   \quad k = 1, \dotsc, s,
   \label{eq:rk:1}
   \\
   & \bu^{n+1} = \bu^n + \Dt{}\, \sum_{k=1}^s b_k \bv_k ,
   \label{eq:rk:2}
\end{align}
where $\bu^n$ and $\bu^{n+1}$ are the approximations of $\bu(t_n)$ and $\bu(t_{n+1})$.
The major cost of finding $\bu^{n+1}$ in the above method is the solution of \cref{eq:rk:1} for $\bv_1, \dotsc, \bv_s$.
Using the Kronecker matrix product $\otimes$ (e.g., see~\cite{Lau05}), the $s\times s$ identity matrix $I_s$, and $\Gamma := {(\gamma_{kj})}_{k,j=1}^s$, the coefficient matrix of \cref{eq:rk:1} can be expressed as
\begin{equation}
 \A \equiv  I_s\otimes M + \Dt{}\, \Gamma \otimes A .
   \label{eq:rk:3}
\end{equation}

The simplest implicit RK method is the implicit Euler method with $s=1$, $\Gamma = 1$, $c_1 = 1$, and $b_1 = 1$
for which \cref{eq:rk:1,eq:rk:2} are reduced to
\begin{equation}
   (M + \Dt{}\, A) \bu^{n+1} = M \bu^n .
   \label{eq:fem:2}
\end{equation}
In \cref{sec:euler} we study the conditioning of the coefficient matrix $M + \Dt{}\, A$ related to the efficient iterative solution of \cref{eq:fem:2}.
The system \cref{eq:fem:2} can also be solved using a symmetric and positive definite preconditioner $P$, which leads to
\begin{equation}
   P^{-\frac{1}{2}} (M + \Dt{}\, A) P^{-\frac{1}{2}} \bw^{n+1}
   = P^{-\frac{1}{2}}MP^{-\frac{1}{2}} \bw^n,
   \qquad \bw^n = P^{\frac{1}{2}}\bu^n.
   \label{eq:fem:3}
\end{equation}
The efficient iterative solution of \cref{eq:fem:3} is related to the conditioning of $P^{-\frac{1}{2}} (M + \Dt{}\, A) P^{-\frac{1}{2}}$.
In this work, we consider as preconditioners the mass matrix $M$, its diagonal part $M_D$, the lumped mass matrix $M_{lump}$, and the diagonal part $M_D + \Dt{}\, A_D$ of $M + \Dt{}\, A$.
The choice $P = M$ is of theoretical importance but impractical since $M^{-\frac{1}{2}}$ is expensive to compute.
The other choices are simple and economic to implement.

The results for the Euler method will be used in \cref{sec:irk} to study the conditioning of \cref{eq:rk:3} and its diagonally preconditioned version for general implicit RK methods.

\section{Preliminary estimates on the extreme eigenvalues
   of the mass and stiffness matrices}\label{sec:preliminary}

Hereafter, $C$ denotes a generic constant which may have different values at different appearances and may depend on the dimension, the choice of the reference element, and the reference basis linear functions but is independent of the mesh and the IBVP coefficients.
For notation simplicity, when using this generic constant $C$, we will sometime write $a \gtrsim b$ and $a \lesssim b$ meaning $a \ge C b$ and $a \le C b$, respectively.

\begin{lemma}[{\cite[proof of Theorem 3.1]{KamHuaXu14}}]\label{lem:M:1}
The mass matrix $M$ and its diagonal part $M_D$ are related by
\begin{equation}
   \frac{1}{2} M_D \leq M \le \left( 1 + \frac{d}{2} \right) M_D
   \quad \text{and} \quad
   M_{jj} = \frac{2 \Abs{\omega_j}}{(d+1)(d+2)} .
   \label{eq:M}
\end{equation}
\end{lemma}

\begin{lemma}[{\cite[Lemma 2.3]{HuaKamLan16}}]\label{lem:M:2}
The mass matrix $M$ and the lumped mass matrix $M_{lump}$ are related by
\begin{equation}
   \frac{1}{d+2} M_{lump} \le M \le \frac{d+2}{2} M_{lump} .
   \label{eq:M-2}
\end{equation}
\end{lemma}

\begin{lemma}[{\cite[Lemma 4.1]{KamHuaXu14}}]\label{lem:A:1}
The stiffness matrix $A$ and its diagonal part $A_D$ satisfy
\begin{equation}
   A \leq (d+1) A_D .
   \label{eq:A}
\end{equation}
\end{lemma}

\begin{lemma}[{\cite[Lemma 2.5]{HuaKamLan16}}]\label{lem:A:1:1}
The diagonal entries of the stiffness matrix are bounded by
\begin{equation}
   C_{\hat\nabla} \sum\limits_{K\in \omega_j}
   \Abs{K} \lambda_{\min}\left(\FDF\right)
   \le A_{jj} \le
   C_{\hat\nabla} \sum\limits_{K\in \omega_j}
   \Abs{K} \lambda_{\max}\left(\FDF\right),
   \label{eq:A:2}
\end{equation}
where
$\D_K = \frac{1}{\Abs{K}} \int_K \D(\bx) \dx$ is the average of $\D$ on $K$
and $C_{\hat\nabla} = \frac{d+1}{d} {\left( \frac{d!}{\sqrt{d+1}} \right)}^{\frac{2}{d}}$.
\end{lemma}

The next two lemmas establish bounds for $A_{jj}$ with a more explicit geometric interpretation than \cref{eq:A:2}.
For this, we denote the diameter and the minimal height of $K$ in the metric $\D_K^{-1}$ by $h_{K, \D^{-1}}$ and
$a_{K, \D^{-1}}$, respectively.
The average element diameter is defined as
\begin{equation}
\hD = {\left(\frac{\Abs{\Omega}_{\D^{-1}}}{N}\right)}^{\frac{1}{d}},
\quad
\Abs{\Omega}_{\D^{-1}} = \sum_{K \in \mathcal{T}_h} \Abs{K} \sqrt{\det(\D_K^{-1})}.
\label{hD-1}
\end{equation}
Further, let $\hat{h}$, $\hat{\rho}$, and $\hat{a}$ be the diameter, the in-diameter, and the minimal height of the unitary equilateral $\hat{K}$, respectively, i.e.,
\[
   \hat{h} =\sqrt{2} {\left( \frac{d!}{\sqrt{d+1}} \right)}^{\frac{1}{d}},
   \quad \hat{\rho} = \sqrt{\frac{2}{d(d+1)}} \hat{h},
   \quad \hat{a} = \sqrt{\frac{d+1}{2 d}} \hat{h} .
\]

\begin{lemma}[{\cite[Lemmas 4.1 and 4.2]{HuaKam15b}}]
It holds
\begin{align}
   & \frac{h_{K,\D^{-1}}^2}{\hat{h}^2}
   \le \norm{{(F_K')}^T \D_K^{-1} F_K'}_2
   \le \frac{h_{K,\D^{-1}}^2}{\hat{\rho}^2}
   ,
   \label{eq:FDF:1}
   \\
   & \frac{\hat{a}^2}{a_{K,\D^{-1}}^2}
   \le \norm{\FDF}_2
   \le \frac{d^2 \hat{a}^2}{a_{K,\D^{-1}}^2}
   .
   \label{eq:FDF:2}
\end{align}
\end{lemma}

\begin{lemma}\label{lem:A:2}
It holds
\begin{equation}
  \hat{\rho}^2 C_{\hat\nabla} \hD^{-2}
   \sum\limits_{K\in \omega_j}\Abs{K}
      {\left( \frac{\hD}{\hKD} \right)}^2
   \le A_{jj} \le
   d^2 \hat{a}^2 C_{\hat\nabla} \hD^{-2}
   \sum\limits_{K\in \omega_j} \Abs{K} {\left( \frac{\hD}{\aKD} \right)}^2 .
   \label{eq:A:3}
\end{equation}
\end{lemma}

\begin{proof}
Combining \cref{eq:FDF:2,eq:A:2} yields
\begin{align*}
   A_{jj} &\le
   d^2 \hat{a}^2 C_{\hat\nabla}
      \sum\limits_{K \in \omega_j} \Abs{K} a_{K, \D^{-1}}^{-2}
   = d^2 \hat{a}^2 C_{\hat\nabla} \hD^{-2}
      \sum\limits_{K \in \omega_j} \Abs{K} {\left( \frac{\hD}{\aKD} \right)}^2
   ,
\end{align*}
which gives the right inequality of \cref{eq:A:3}.
On the other hand, \cref{eq:FDF:1} gives
\[
   \lambda_{\min}(\FDF)
   = \norm{{(F_K')}^T \D_K^{-1} F_K'}_2^{-1}
   \ge \hat{\rho}^2 h_{K,\D^{-1}}^{-2}
   .
\]
Combining this with \cref{eq:A:2} leads to the left inequality of \cref{eq:A:3}.
\end{proof}

A mesh that is uniform in the metric $\D^{-1}$ (\emph{a $\D^{-1}$-uniform mesh}) satisfies $h_{K,\D^{-1}} \sim a_{K,\D^{-1}} \sim h_{\D^{-1}}$.
In such a case, \cref{eq:A:3} implies $A_{jj} \sim \Abs{\omega_j} \hD^{-2} \sim \Abs{\omega_j} N^{\frac{2}{d}}$.

\begin{lemma}[{\cite[Lemma~5.1]{KamHuaXu14}}]\label{lem:A:3}
The smallest eigenvalue of the stiffness matrix is bounded by
\begin{equation}
   \lambda_{\min}(A)  \gtrsim d_{\min} N^{-1} \EnonU^{-1},
\label{eq:khx2014}
\end{equation}
where
\begin{equation}
   \EnonU =
   \begin{cases}
      1,
      & d = 1,
      \\
      1 + \ln \left( \frac{\abs{\bar{K}}}{\Abs{K_{\min}}} \right),
      & d = 2,
      \\
      {\left( \frac{1}{N} \sum\limits_{K\in\Th}
         {\left( \frac{\abs{\bar{K}}}{\Abs{K}} \right)}^{\frac{d-2}{2}}
            \right)}^{\frac{2}{d}},
      & d\geq3,
   \end{cases}
   \label{eq:alpha-1}
\end{equation}
$\Abs{K_{\min}}$ and $\abs{\bar{K}} = \frac{1}{N}\abs{\Omega}$ are the minimal and the average element volumes and $d_{\min}$ is the global smallest eigenvalue of $\D$ (cf.\ \cref{eq:D:i}).
\end{lemma}

The factor $\EnonU$ shows that the dependence of the lower bound on the mesh non-uniformity is mild although getting stronger in higher dimensions: $\EnonU$ is mesh-independent in 1d, contains only $\ln \left ({\abs{\bar{K}}}/{\Abs{K_{\min}}}\right )$ in 2d, and involves the generalized $d/2$-mean of the ratio ${\left({\abs{\bar{K}}}/{\Abs{K}}\right)}^{(d-2)/{d}}$ over all elements for $d\ge 3$.

The next lemmas summarize bounds for the extremal eigenvalues of $M$ and $A$.

\begin{lemma}\label{lem:M-maxmin}
The extremal eigenvalues of $P^{-\frac{1}{2}} MP^{-\frac{1}{2}} $ are bounded by
\begin{equation}
   \lambda_{\max}(P^{-\frac{1}{2}}MP^{-\frac{1}{2}})
   \le \begin{cases}
      \frac{\abs{\omega_{\max}}}{d+1},
   & P = I_{N_{vi}},
   \\
   1 + \frac{d}{2},
   & P = M_{D}, ~ M_{lump}, ~ M_D + \Dt{}\, A_{D},
   \end{cases}
   \label{M-max}
\end{equation}
and
\begin{equation}
   \lambda_{\min}(P^{-\frac{1}{2}}MP^{-\frac{1}{2}})
   \ge \begin{cases}
      \frac{\abs{\omega_{\min}}}{(d+1)(d+2)},
      & P = I_{N_{vi}},
   \\
   1/2,
   & P = M_D,
   \\
   1/(d+2),
   & P = M_{lump},
   \\
   C {\left(1 + \Dt{}\, \hD^{-2}
   \max\limits_{j} \sum\limits_{K\in \omega_j} \frac{\Abs{K}}{\Abs{\omega_j}}
      {\left(\frac{\hD}{\aKD}\right)}^2 \right)}^{-1}
      & P = M_D + \Dt{}\, A_D,
   \end{cases}
   \label{M-min}
\end{equation}
where $\Abs{\omega_{\max}}$ and $\Abs{\omega_{\min}}$ are the maximal and minimal patch volumes, respectively.
\end{lemma}

\begin{proof}
The inequalities \cref{M-max} and \cref{M-min} for $P = I_{N_{vi}}$ (no preconditioning), $P = M_D$, and $M_{lump}$ follow from \cref{lem:M:1,lem:M:2}.
For $P = M_D + \Dt{}\, A_D$, \cref{eq:M,eq:A:3} give
\begin{equation}
   P_{jj}
   \lesssim \Abs{\omega_j}
      + \Dt{}\, \hD^{-2}
   \sum\limits_{K \in \omega_j}\Abs{K} {\left( \frac{h_{\D^{-1}}}{a_{K, \D^{-1}}} \right)}^2
   .
   \label{eq:Pkk-0}
\end{equation}
Then, \cref{M-min} follows from
\[
   \bv^T P^{-\frac{1}{2}} M P^{-\frac{1}{2}} \bv
      \gtrsim \sum\limits_j \frac{v_j^2 \Abs{\omega_j}}{P_{jj}}
      \gtrsim \sum\limits_j v_j^2
         {\left( 1 + \Dt{}\, \hD^{-2}
      \sum\limits_{K \in \omega_j} \frac{\Abs{K}}{\Abs{\omega_j}} {\left( \frac{h_{\D^{-1}}}{a_{K, \D^{-1}}} \right)}^2 \right)}^{-1}.
\qedhere{}
\]
\end{proof}

\begin{lemma}\label{lem:A-maxmin}
The extremal eigenvalues of $P^{-\frac{1}{2}} A P^{-\frac{1}{2}} $ are bounded by
\begin{equation}
   \lambda_{\max}(P^{-\frac{1}{2}} A P^{-\frac{1}{2}})
   \le \begin{cases}
      C \hD^{-2}
      \max\limits_j \sum\limits_{K\in \omega_j} \Abs{K} {\left(\frac{\hD}{\aKD}\right)}^2,
      & P = I_{N_{vi}},
      \\
      C \hD^{-2}
      \max\limits_{j}
         \sum\limits_{K\in \omega_j} \frac{\Abs{K}}{\Abs{\omega_j}} {\left(\frac{\hD}{\aKD}\right)}^2,
         & P = M,~M_D,~M_{lump},
      \\
      (d+1) \Dt{}^{-1},
      & P = M_D + \Dt{}\, A_D,
   \end{cases}
   \label{eq:A:max}
\end{equation}
and
\begin{equation}
   \lambda_{\min}(P^{-\frac{1}{2}}A P^{-\frac{1}{2}})
   \gtrsim \begin{cases}
      d_{\min} N^{-1} \EnonU^{-1},
      & P = I_{N_{vi}},
      \\
      \lambda_{\D},
      & P = M,~M_D,~M_{lump},
      \\
      d_{\min} \DnonU^{-1},
      & P = M_D + \Dt{}\, A_D,
   \end{cases}
   \label{eq:A:min}
\end{equation}
where $\lambda_{\D}$ is the minimal eigenvalue of the operator $-\nabla \cdot (\D \nabla )$ and
\begin{align}
   \DnonU &=
   \begin{cases}
      1 +  \Dt{}\, \hD^{-2}
         \sum\limits_{K} \Abs{K} {\left(\frac{\hD}{\aKD}\right)}^2,
      & d = 1,
      \\
      \left(1+ \Abs{\ln \nonUinLog}\right)
         \left(1 + \Dt{}\, \hD^{-2}
         \sum\limits_{K} \Abs{K} {\left(\frac{\hD}{\aKD}\right)} ^2\right),
      & d = 2,
      \\
      1 + \Dt{}\, \hD^{-2}
         {\left(\sum\limits_K \abs{K} {\left( \frac{\hD}{\aKD} \right)}^d \right)}^{\frac{2}{d}},
      & d \ge 3,
   \end{cases}
   \label{eq:beta-1}
   \\
   \nonUinLog &=
   \frac{1 + \Dt{}\, \hD^{-2}
      \max\limits_K {\left(\frac{\hD}{\aKD}\right)}^2}
   {1 + \Dt{}\, \hD^{-2}
   \sum\limits_{K} \frac{\Abs{K}}{\Abs{\Omega}} {\left(\frac{\hD}{\aKD}\right)}^2} .
   \label{eq:gamma-1}
\end{align}
\end{lemma}

\begin{proof}
For $P = I_{N_{vi}}$, \cref{eq:A:max} follows from \cref{lem:A:1,lem:A:2}.
For $P = M$, \cref{eq:M,eq:A} give
\[
   \lambda_{\max}(P^{-\frac{1}{2}} A P^{-\frac{1}{2}} )
   = \max_{\bv \neq 0} \frac{\bv^T A \bv}{\bv^T P \bv}
   \le 2(d+1) \max_{\bv \neq 0} \frac{\bv^T A_D \bv}{\bv^T M_D \bv}
   = 2(d+1) \max\limits_j \frac{A_{jj}}{M_{jj}}
   ,
 \]
which, together with \cref{eq:M,eq:A:3}, implies \cref{eq:A:max}.
For $P=M_D$ and $P=M_{lump}$, \cref{eq:A:max} follows from \cref{lem:M:1,lem:M:2}.
For $P = M_D + \Dt{}\, A_D$, \cref{eq:A:max} follows from \cref{eq:A}.

The inequality \cref{eq:A:min} for $P = I_{N_{vi}}$ follows from \cref{lem:A:3}.
For $P=M$ and, similarly for $P=M_D$ and $P=M_{lump}$, \cref{eq:A:min} follows from the basic property of the conformal FE approximation of elliptic eigenvalue problems (e.g.,~\cite[Theorem 1]{Fix73}), since the eigenvalue problem for $P^{-\frac{1}{2}}A P^{-\frac{1}{2}}$ is a FE approximation to the eigenvalue problem for the operator $-\nabla \cdot (\D \nabla )$.

For $P = M_D + \Dt{}\, A_D$, using the strategy in the proof of~\cite[Lemma~5.1]{KamHuaXu14}, we have
\begin{equation}
   \bv^T P^{-\frac{1}{2}} A P^{-\frac{1}{2}} \bv
   \gtrsim d_{\min}
   \times \begin{cases}
      \left(\sum\limits_j v_j^2\right)
      \cdot {\left( \sum\limits_j P_{jj} \right)}^{-1},
      & d = 1,\\
      \frac{1}{q} {\left(\sum\limits_K s_K^{\frac{q}{q-2}}
         \right)}^{- \frac{q-2}{q}} \sum\limits_j v_j^2 P_{jj}^{-1}
      \sum\limits_{K \in \omega_j} s_K \Abs{K}^{\frac{2}{q}},
      & d = 2,\\
      {\left(\sum\limits_K s_K^{\frac{d}{2}} \right)}^{- \frac{2}{d}}
         \sum\limits_j v_j^2 P_{jj}^{-1}
      \sum\limits_{K \in \omega_j} s_K \Abs{K}^{\frac{d-2}{d}},
      & d\geq3,
   \end{cases}
   \label{eq:implicit:Euler:lmin:scaled:i}
\end{equation}
with a parameter $q > 2$ and some not-all-zero nonnegative numbers $s_K$, $K \in \Th$ (to be determined later).
Here, $\sum\limits_{K}$ denotes the summation over all elements in $\Th$.

For notational simplicity, we denote $r_K = \aKD^{-2}$ and rewrite \cref{eq:Pkk-0} as
\begin{equation}
   P_{jj} \lesssim \sum\limits_{K \in \omega_j}\abs{K} (1+ \Dt{}\, r_K) .
\label{eq:Pjj}
\end{equation}

For $d=1$, \cref{eq:Pjj} implies
\[
   \bv^T P^{-\frac{1}{2}} A P^{-\frac{1}{2}}\bv
   \gtrsim  \frac{d_{\min} \sum\limits_j v_j^2}
   {\sum\limits_{K} \Abs{K} \left(1 + \Dt{}\, r_K \right)}
\]
and therefore
\begin{equation}
     \lambda_{\min}(P^{-\frac{1}{2}} A P^{-\frac{1}{2}})
      \gtrsim \frac{d_{\min} }
      {\sum\limits_{K} \Abs{K} \left(1 + \Dt{}\, r_K \right) }.
   \label{eq:A:4:1d}
\end{equation}

For $d=2$, we choose $s_K = \Abs{K}^{1 - \frac{2}{q}} (1 + \Dt{}\, r_K)$.
Then, \cref{eq:Pjj} implies
\[
   P_{jj}^{-1} \sum\limits_{K \in \omega_j} s_K \Abs{K}^{\frac{2}{q}} \ge C
\]
and using \cref{eq:implicit:Euler:lmin:scaled:i} we obtain
\[
   \bv^T P^{-\frac{1}{2}} A P^{-\frac{1}{2}} \bv
      \gtrsim \frac{d_{\min} \sum\limits_k v_k^2}
         {q {\left(\sum\limits_K \Abs{K} {(1+\Dt{}\, r_K)}^{\frac{q}{q-2}}
            \right)}^{ \frac{q-2}{q}} } .
\]
The denominator tends to infinity as $q\to \infty$ and to $2 \max_{K} (1 + \Dt{}\, r_K)$ as $q \to 2^{+}$.
To find an optimal choice for $q$, we first estimate the denominator as
\begin{align*}
   q {\left(\sum\limits_K \Abs{K} {\left(1+\Dt{}\, r_K\right)}^{\frac{q}{q-2}}
      \right)}^{ \frac{q-2}{q}}
   &= q {\left(\sum\limits_K \Abs{K} (1+\Dt{}\, r_K)
      \cdot {(1+\Dt{}\, r_K)}^{\frac{2}{q-2}} \right)}^{\frac{q-2}{q}}
   \\
   &\le q {\left(\left(\sum\limits_K \Abs{K} (1+\Dt{}\, r_K)\right)
      \cdot {\left(1+\Dt{}\, \max\limits_K r_K\right)}^{\frac{2}{q-2}}
         \right)}^{ \frac{q-2}{q}}
   \\
   &= q \nonUinLog^{\frac{2}{q}} \sum\limits_K \Abs{K} (1+\Dt{}\, r_K),
\end{align*}
where
\begin{equation}
   \nonUinLog = \frac{ 1 + \Dt{}\, \max\limits_K r_K}
      {\sum\limits_K \Abs{K} (1+\Dt{}\, r_K)} .
   \label{eq:gamma-0}
\end{equation}
If $\nonUinLog > e$, we choose $q = 2 \ln \nonUinLog$ and obtain $q \nonUinLog^{\frac{2}{q}} = 2 e \ln \nonUinLog$.
If $\nonUinLog \le e$, we use $q \to 2^{+}$ and obtain $q \nonUinLog^{\frac{2}{q}} \le 2 e$.
Combining these two cases yields
\[
   q {\left(\sum\limits_K \Abs{K} {(1+\Dt{}\, r_K)}^{\frac{q}{q-2}}
      \right)}^{ \frac{q-2}{q}}
   \le 2 e \left(1+ \Abs{\ln \nonUinLog}\right)
      \left(\sum\limits_K \Abs{K} (1+\Dt{}\, r_K)\right)
\]
and therefore
\begin{equation}
   \lambda_{\min}(P^{-\frac{1}{2}} A P^{-\frac{1}{2}})
   \gtrsim  \frac{d_{\min}}
         {\left(1+ \Abs{\ln \nonUinLog}\right)
         \left(\sum\limits_K \Abs{K} (1+\Dt{}\, r_K)\right) } .
   \label{eq:A:4:2d}
\end{equation}

For $d\geq3$, we choose $s_K = \Abs{K}^{\frac{2}{d}} (1+\Dt{}\, r_K)$ such that $P_{jj}^{-1} \sum\limits_{K \in \omega_j} r_K \Abs{K}^{\frac{2}{q}} \ge C$.
From \cref{eq:implicit:Euler:lmin:scaled:i}, we have
\[
   \bv^T P^{-\frac{1}{2}} A P^{-\frac{1}{2}} \bu
      \gtrsim \frac{d_{\min} \sum\limits_j v_j^2}
      {{\left(\sum\limits_{K} \Abs{K}
         {\left(1 + \Dt{}\, r_K \right)}^{\frac{d}{2}} \right)}^{\frac{2}{d}} } ,
\]
which gives
\begin{equation}
   \lambda_{\min}(P^{-\frac{1}{2}} A P^{-\frac{1}{2}})
      \gtrsim  \frac{d_{\min} }
            { {\left(\sum\limits_{K} \Abs{K}
               {\left(1 + \Dt{}\, r_K \right)}^{\frac{d}{2}}
                  \right)}^{\frac{2}{d}} } .
\label{eq:A:4:3d}
\end{equation}

Further, \cref{eq:gamma-0} can be rewritten into \cref{eq:gamma-1} and
\begin{align*}
   \sum\limits_{K} \Abs{K} (1 + \Dt{}\, r_K)
   & = \abs{\Omega} +  \Dt{}\, \hD^{-2}
   \sum\limits_{K} \Abs{K} {\left(\frac{\hD}{\aKD}\right)}^2
   ,
   \\
   {\left( \sum\limits_K \abs{K} {(1+\Dt{}\, r_K )}^{\frac{d}{2}} \right)}^{\frac{2}{d}}
   & \lesssim \Abs{\Omega}^{\frac{2}{d}} + \Dt{}\, \hD^{-2}
      {\left( \sum\limits_K \abs{K} {\left(\frac{\hD}{\aKD}\right)}^d
      \right)}^{\frac{2}{d}}
   .
\end{align*}
Combining these with \crefrange{eq:A:4:1d}{eq:A:4:3d} gives \cref{eq:A:min} with $P = M_D + \Dt{}\, A_D$.
\end{proof}

\section{Conditioning of the implicit Euler integration}\label{sec:euler}

In the following we estimate the conditioning of $M + \Dt{}\, A$ (in the $l_2$-norm), which is related to the efficient iterative solution of \cref{eq:fem:2}.

\subsection{%
   \texorpdfstring{Condition number of $M + \Dt{}\, A$}
   {Condition of the system matrix}}\label{sec:condition:euler}

\begin{theorem}\label{thm:euler:1}
The condition number of $M + \Dt{}\, A$ is bounded by
\begin{equation}
   \kappa(M + \Dt{}\, A)
      \lesssim \frac{\max\limits_j  \frac{\abs{\omega_j}}{\abs{\omega_{\min}}} \left (1
      + \Dt{}\, \hD^{-2}
      \sum\limits_{K \in \omega_j} \frac{\Abs{K}}{\Abs{\omega_j}} {\left(\frac{\hD}{\aKD} \right)}^2\right )}
      {1 +  \Dt{}\, d_{\min} N^{-1} \Abs{\omega_{\min}}^{-1} \EnonU^{-1}} .
   \label{eq:cond:1}
\end{equation}
\end{theorem}

\begin{proof}
\Cref{eq:M,eq:A} yield
\[
   M + \Dt{}\, A
   \leq \frac{d+2}{2} M_D  + \Dt{}\, (d+1) A_D
   \le (d+1)  (M_D + \Dt{}\, A_D)
   .
\]
From this and \cref{eq:M,eq:A:3} we obtain
\[
   \lambda_{\max}(M + \Dt{}\, A)
   \lesssim \max\limits_j
   \left( \Abs{\omega_j}
   + \Dt{}\, \hD^{-2} \sum\limits_{K \in \omega_j} \Abs{K} {\left(\frac{\hD}{\aKD} \right)}^2 \right)
   .
\]

On the other hand, using $ \lambda_{\min}(M + \Dt{}\, A) \ge \lambda_{\min}(M) + \Dt{}\, \lambda_{\min}(A)$
and \cref{lem:M-maxmin,lem:A-maxmin} we have
\[
   \lambda_{\min}(M + \Dt{}\, A)
   \gtrsim \Abs{\omega_{\min}} +  \Dt{}\, d_{\min} N^{-1} \EnonU^{-1}.
\]
Combining the above results yields \cref{eq:cond:1}.
\end{proof}

There are three factors influencing bound \cref{eq:cond:1}.
The first factor is the number of the mesh elements $N$.
The second factor is
\begin{equation}
   \sum\limits_{K \in \omega_j} \frac{\Abs{K}}{\Abs{\omega_j}}  {\left( \frac{\hD}{\aKD} \right)}^2,
   \label{eq:dnonuniformity:factor}
\end{equation}
which reflects the mesh nonuniformity in the metric $\D^{-1}$ and is a constant for a $\D^{-1}$-uniform mesh, for which $\aKD \sim \hD$ for all $K$.
The third factor is the effect of the mesh nonuniformity (in the Euclidean metric) reflected by $\EnonU$, ${\Abs{\omega_j}}/{\Abs{\omega_{\min}}}$, and $N \Abs{\omega_{\min}}$,
which all become constants if the mesh is uniform.

The time step size $\Dt{}$ plays the role of a homotopy parameter between the mass and the stiffness matrices:
\begin{equation}
   \kappa(M + \Dt{}\, A)
   ~ \xrightarrow{\Dt{}\, \to 0} ~
      \kappa(M) \lesssim \frac{\Abs{\omega_{\max}}}{\Abs{\omega_{\min}}} .
   \label{eq:cond:1a}
\end{equation}
Thus, $\kappa(M)$ depends on the nonuniformity of the element patch volumes.
On the other hand,
\begin{align}
   \kappa(M + \Dt{}\, A)
   ~ \xrightarrow{\Dt{}\, \to \infty} ~
       \kappa(A)
   \lesssim d_{\min}^{-1} \hD^{-2} \EnonU
   \max\limits_j \left (N \sum\limits_{K \in \omega_j}
   \abs{K} {\left(\frac{\hD}{\aKD}\right)}^2 \right)
   ,
  \label{eq:cond:1b}
\end{align}
which has been obtained previously~\cite[Theorem 5.2]{KamHuaXu14}.
It depends on three factors as well: $N$ (through $\hD$), the mesh nonuniformity in the Euclidean metric (through $\EnonU$), and the mesh nonuniformity with respect to $\D^{-1}$ (through the term in the maximum).

\subsection{%
\texorpdfstring{Condition number of $P^{- \frac{1}{2}} (M + \Dt{}\, A)P^{- \frac{1}{2}}$}
   {Condition number of the preconditioned system matrix}}\label{sec:euler:scaled}

\begin{theorem}\label{thm:euler:2}
The condition number of $P^{- \frac{1}{2}} (M + \Dt{}\, A)P^{- \frac{1}{2}}$ with $P = M$, $M_D$, or $M_{lump}$ is bounded by
\begin{equation}
   \kappa(P^{-\frac{1}{2}} (M + \Dt{}\, A) P^{-\frac{1}{2}})
   \lesssim \frac{1 + \Dt{}\, \hD^{-2}
      \max\limits_j  \sum\limits_{K \in \omega_j} \frac{\abs{K}}{\Abs{\omega_j}}
      {\left(\frac{\hD}{\aKD} \right)}^2}
   { 1 + \Dt{}\, \lambda_{\D}}
   ,
   \label{eq:cond:2}
\end{equation}
where $\lambda_{\D}$ is the minimal eigenvalue of $-\nabla \cdot (\D \nabla )$.
\end{theorem}

\begin{proof}
Bound \cref{eq:cond:2} is obtained by combining \cref{lem:M-maxmin,lem:A-maxmin} with the estimates
\begin{align*}
& \lambda_{\max}\left ( P^{-\frac{1}{2}} (M + \Dt{}\, A) P^{-\frac{1}{2}} \right )
\le \lambda_{\max} \left ( P^{-\frac{1}{2}} M P^{-\frac{1}{2}} \right )
+ \Dt{}\, \lambda_{\max} \left ( P^{-\frac{1}{2}} A P^{-\frac{1}{2}} \right ),
\\
& \lambda_{\min}\left ( P^{-\frac{1}{2}} (M + \Dt{}\, A) P^{-\frac{1}{2}} \right )
\ge \lambda_{\min} \left ( P^{-\frac{1}{2}} M P^{-\frac{1}{2}} \right )
+ \Dt{}\, \lambda_{\min} \left ( P^{-\frac{1}{2}} A P^{-\frac{1}{2}} \right )
.
\qedhere{}
\end{align*}
\end{proof}

The bounds on $\kappa(P^{-\frac{1}{2}} (M + \Dt{}\, A) P^{-\frac{1}{2}})$ for $P = M$, $M_D$, and $M_{lump}$ are similar, while the last two choices lead to a simple diagonal scaling, which is easier to implement.

All three choices reduce the effects of the mesh nonuniformity: in comparison to \cref{eq:cond:1}, bound \cref{eq:cond:2} does not depend on $\Abs{\omega_{\min}}$ or $\EnonU$ directly and contains only the $\D^{-1}$-nonuniformity factor \cref{eq:dnonuniformity:factor}.
This is intuitive, since the eigenvalues of $M^{-1} A$ approximate those of the underlying continuous operator, which are mesh-independent.
However, $\kappa(M^{-1} A)$ is not necessarily smaller then $\kappa(A)$ and the overall effect depends on the magnitude of $\Dt{}$: if $\Dt{}$ is large, $\kappa(M^{-1}(M + \Dt{} A))$ might not be better than $\kappa(M + \Dt{} A)$.
On the other hand, if $\Dt{}$ is small, we can expect that $\kappa(M^{-1}(M + \Dt{} A)) < \kappa(M + \Dt{} A)$.
The numerical experiments in \cref{sec:numerics} support this argument (see \cref{ex:scaling,fig:scaling:MD}).

\begin{theorem}\label{thm:euler:3}
The condition number of $P^{-\frac{1}{2}} (M + \Dt{}\, A) P^{-\frac{1}{2}}$
with the Jacobi preconditioner
$P = M_D + \Dt{}\, A_D$ is bounded by
\begin{align}
    \kappa(P^{-\frac{1}{2}} (M + \Dt{}\, A) P^{-\frac{1}{2}})
   \lesssim
   {\left(
      \frac{1}{1 + \Dt{}\, \hD^{-2} \max\limits_j \sum\limits_{K \in \omega_j}
         \frac{\Abs{K}}{\Abs{\omega_j}} {\left(\frac{\hD}{\aKD}\right)}^2}
      + \frac{d_{\min}}{\DnonU} \right)}^{-1}
   \label{eq:cond:3}
\end{align}
where $\DnonU$ is given in \cref{eq:beta-1}.
\end{theorem}

\begin{proof}
The proof is similar to that of \cref{thm:euler:2}.
\end{proof}

Bound \cref{eq:cond:3} is comparable to \cref{eq:cond:2} although the former is smaller than the latter in general, especially for large $\Dt{}$, since the factor $\DnonU$ in \cref{eq:cond:3} involves averaging over all elements, whereas \cref{eq:cond:2} involves the maximum over patch averages.

In \cref{eq:cond:3}, $\Dt{}$ is a homotopy parameter between the mass and stiffness matrices:
\[
   \kappa(P^{-\frac{1}{2}} (M + \Dt{}\, A) P^{-\frac{1}{2}})
   ~ \xrightarrow{\Dt{}\, \to 0} ~
   \kappa (M_D^{-\frac{1}{2}} M M_D^{-\frac{1}{2}}) \le C .
\]
On the other hand,
\begin{multline}
   \kappa(P^{-\frac{1}{2}} (M + \Dt{}\, A) P^{-\frac{1}{2}})
    \xrightarrow{\Dt{}\, \to \infty}
   \kappa (A_D^{-\frac{1}{2}} A A_D^{-\frac{1}{2}})
   \notag
   \\
   \lesssim \frac{N^{\frac{2}{d}}}{d_{\min} \Abs{\Omega}_{\D^{-1}}^{\frac{2}{d}}}
   \times \begin{cases}
      \sum\limits_{K} \Abs{K}  {\left( \frac{\hD}{\aKD} \right)}^2,
      &d = 1,
      \\
      {\left( 1+ \Abs{\ln \nonUinLog}\right) \sum\limits_{K} \Abs{K}  \left (\frac{\hD}{\aKD} \right)}^2,
      &d=2,
      \\
      {\left(\sum\limits_{K} \Abs{K} {\left( \frac{\hD}{\aKD} \right)}^d \right)}^{\frac{2}{d}},
      &d\geq3,
   \end{cases}
   \label{eq:cond:4}
\end{multline}
where $\nonUinLog$ becomes
\[
   \nonUinLog =
   \frac{\max\limits_{K} {\left( \frac{\hD}{\aKD} \right)}^2}
   {\sum\limits_{K} \frac{\Abs{K}}{\Abs{\Omega}} {\left( \frac{\hD}{\aKD} \right)}^2}
   .
\]
This bound is equivalent to the bound obtained in~\cite[Theorem 5.2]{KamHuaXu14}.

\begin{remark}\label{rem:scaling}
In comparison to the bound \cref{eq:cond:1} for $M + \Dt{}\, A$, bounds \cref{eq:cond:2,eq:cond:3} contain only two mesh-dependent factors: $N$ and the mesh nonuniformity in the metric $\D^{-1}$ (through the terms involving the ratio $\hD/\aKD$).
This shows that the effects of the mesh nonuniformity (in the Euclidean metric) on the condition number is effectively eliminated by the preconditioning.
\end{remark}

\begin{remark}\label{rem:dt:h}
For a $\D^{-1}$-uniform mesh, $\aKD \sim \hD$ and, hence, all terms involving the ratio $\hD/\aKD$ will become a constant.
Thus,
\[
   \kappa(P^{-\frac{1}{2}} (M + \Dt{}\, A) P^{-\frac{1}{2}})  = \mathcal{O}(1+\Dt{}\, N^{\frac{2}{d}}),
\]
which shows more clearly the role of $\Dt{}$:
\[
   \kappa(P^{-\frac{1}{2}} (M + \Dt{}\, A) P^{-\frac{1}{2}})  =
   \begin{cases}
      \mathcal{O}(N^{\frac{2}{d}}), &\Dt{}\, = \mathcal{O}(1),
      \\
      \mathcal{O}(N^{\frac{1}{d}}), &\Dt{}\, = \mathcal{O}(\hD) = \mathcal{O}(N^{-\frac{1}{d}}),
      \\
      \mathcal{O}(1), &\Dt{}\, = \mathcal{O}(N^{-\frac{2}{d}})
      \text{ or $\Dt{}\, \to 0$}.
   \end{cases}
\]
\end{remark}

\section{Conditioning of general implicit RK integration}\label{sec:irk}

For a general implicit RK method, the matrix $\A$ in \cref{eq:rk:3} is not necessarily normal
and in this case a condition number in the standard definition does not provide much
information for the convergence of iterative methods. When its symmetric part, $(\mathbb{A} + \mathbb{A}^T)/2$,
is positive definite, it is known~\cite{EisElmSch83} that the convergence of the generalized minimal residual
method (GMRES) is
\begin{equation}
\| r_n \| \le {\left(1 - \frac{\lambda_{\min}^2((\mathbb{A} + \mathbb{A}^T)/2)}{\sigma_{\max}^2(\mathbb{A})}
\right)}^{n/2} \| r _0\|,
\label{Eisenstat-1}
\end{equation}
where $\sigma_{\max}(\mathbb{A})$ is the largest singular value of $\mathbb{A}$,
$\lambda_{\min}((\mathbb{A} + \mathbb{A}^T)/2)$
is the minimal eigenvalue of the symmetric part, $r_n$ is the residual
of the corresponding linear system at the $n$-th iterate, and $\| \cdot \|$ stands for the matrix or vector 2-norm.
Thus, we can consider the ``condition number''
\begin{equation}
   \tilde{\kappa}(\mathbb{A}) = \frac{\sigma_{\max}(\mathbb{A})}{\lambda_{\min}((\mathbb{A} + \mathbb{A}^T)/2)} .
   \label{cond-0}
\end{equation}
This definition reduces to the standard definition of the condition number (in 2-norm)
for symmetric matrices.

We first recall that for an arbitrary square matrix $U$, its maximal singular value is given by
\[
   \sigma_{\max}(U) = \max_{\bv \neq 0} \frac{ \norm{ U \bv }}{\norm{ \bv }} .
\]
From this, for any square matrices $U$ and $V$ we have
\begin{equation}
   \sigma_{\max}(U+V) \le \sigma_{\max}(U) + \sigma_{\max}(V) .
   \label{sv-1}
\end{equation}

The next lemma provides the eigenvalues and singular values of the Kronecker product of any square matrices $U$ and $V$.

\begin{lemma}[{\cite[Theorems 13.10 and 13.12]{Lau05}}]\label{lem:sv:2}
For any square matrices $U$ and $V$, the eigenvalues and singular values of $U\otimes V$ are $\lambda_j(U)\lambda_k(V)$ and $\sigma_j(U) \sigma_k(V)$, $j,k = 1, 2, \dotso$
\end{lemma}

The above lemma in particular implies
\begin{equation*}
   \sigma_{\max}(U\otimes V) = \sigma_{\max}(U) \sigma_{\max}(V)
   \label{eq:kronSigmaMax}
\end{equation*}
and, for a positive semi-definite matrix $V$,
\begin{equation}
   \lambda_{\min}(U\otimes V) = \lambda_{\min}(U) \times
      \begin{cases}
         \lambda_{\min}(V), & \text{if $\lambda_{\min}(U) \ge 0$},\\
         \lambda_{\max}(V), & \text{if $\lambda_{\min}(U) < 0$}.
      \end{cases}
   \label{eq:kronLmin}
\end{equation}

\subsection{General implicit RK methods: simultaneous solution}\label{sec:irk-1}

Although a successive strategy for implicit RK methods is more common (e.g.,~\cite[p.~131]{HaiWan96}), the simultaneous solution can be appealing if an iterative method is used.
It is also important to have a theoretical understanding of the overall system conditioning.

\begin{theorem}\label{thm:irk:1}
Let $P$ be a symmetric and positive definite preconditioner, $\Gamma$ an $s \times s$ implicit RK coefficient matrix, and $\A = I_s \otimes M + \Dt \, \Gamma A$.

If $\lambda_{\min}((\Gamma + \Gamma^T)/2) < 0$ and $\Dt$ is chosen sufficiently small such that
\begin{equation}
   \lambda_{\min}(P^{-\frac{1}{2}} M P^{-\frac{1}{2}}) 
      + \Dt \lambda_{\min}((\Gamma+\Gamma^T)/2)
            \lambda_{\max}(P^{-\frac{1}{2}} A P^{-\frac{1}{2}}) > 0,
 \label{dt-0}
\end{equation}
then
\begin{multline}
   \tilde{\kappa} \left( (I_s\otimes P^{-\frac{1}{2}}) \A
      (I_s\otimes P^{-\frac{1}{2}}) \right)
   \\
   \le
   \frac{\lambda_{\max}(P^{-\frac{1}{2}} M P^{-\frac{1}{2}})
            + \Dt{}\, \sigma_{\max}(\Gamma)
            \lambda_{\max}(P^{-\frac{1}{2}} A P^{-\frac{1}{2}})}
        {\lambda_{\min}(P^{-\frac{1}{2}} M P^{-\frac{1}{2}}) 
             + \Dt \lambda_{\min}((\Gamma+\Gamma^T)/2)
             \lambda_{\max}(P^{-\frac{1}{2}} A P^{-\frac{1}{2}})
       } .
   \label{eq:irk:1:1}
\end{multline}

If $\lambda_{\min}((\Gamma + \Gamma^T)/2) \ge 0$, then
\begin{multline}
   \tilde{\kappa} \left( (I_s\otimes P^{-\frac{1}{2}}) \A
      (I_s\otimes P^{-\frac{1}{2}}) \right)
   \\
   \le \frac{\lambda_{\max}(P^{-\frac{1}{2}} M P^{-\frac{1}{2}}) + \Dt{}\, \sigma_{\max}(\Gamma)
            \lambda_{\max}(P^{-\frac{1}{2}} A P^{-\frac{1}{2}})}
         {\lambda_{\min}(P^{-\frac{1}{2}} M P^{-\frac{1}{2}}) 
            + \Dt \lambda_{\min}((\Gamma+\Gamma^T)/2)
            \lambda_{\min}(P^{-\frac{1}{2}} A P^{-\frac{1}{2}})}
        .
   \label{eq:irk:1:0}
\end{multline}

\end{theorem}

\begin{proof}
First,  notice that
\[
   (I_s\otimes P^{-\frac{1}{2}}) (I_s\otimes M+\Dt{}\, \Gamma \otimes A) (I_s\otimes P^{-\frac{1}{2}})
   = I_s\otimes (P^{-\frac{1}{2}}MP^{-\frac{1}{2}}) +\Dt{}\, \Gamma \otimes (P^{-\frac{1}{2}}AP^{-\frac{1}{2}}) .
\]
Then, from \cref{sv-1,lem:sv:2} we have
\begin{align}
& \sigma_{\max}\left ((I_s\otimes P^{-\frac{1}{2}}) (I_s\otimes M+\Dt{}\, \Gamma \otimes A) (I_s\otimes P^{-\frac{1}{2}}) \right )
\notag
\\
& \quad\qquad \le \sigma_{\max}\left ( I_s\otimes (P^{-\frac{1}{2}}MP^{-\frac{1}{2}}) \right )
+ \Dt{}\, \sigma_{\max}\left ( \Gamma \otimes (P^{-\frac{1}{2}}AP^{-\frac{1}{2}}) \right )
\notag
\\
& \quad\qquad = \sigma_{\max}\left ( P^{-\frac{1}{2}}MP^{-\frac{1}{2}} \right )
+ \Dt{}\, \sigma_{\max}\left ( \Gamma \right ) \sigma_{\max}\left ( P^{-\frac{1}{2}}AP^{-\frac{1}{2}} \right )
\notag
\\
& \quad\qquad = \lambda_{\max}\left ( P^{-\frac{1}{2}}MP^{-\frac{1}{2}} \right )
+ \Dt{}\, \sigma_{\max}\left ( \Gamma \right ) \lambda_{\max}\left ( P^{-\frac{1}{2}}AP^{-\frac{1}{2}} \right ) .
\label{eq:irk-1-2}
\end{align}
Moreover, the symmetric part of
$(I_s\otimes P^{-\frac{1}{2}}) (I_s\otimes M+\Dt{}\, \Gamma \otimes A) (I_s\otimes P^{-\frac{1}{2}})$ is
\[
I_s\otimes (P^{-\frac{1}{2}}MP^{-\frac{1}{2}}) +\Dt{}\, (\Gamma+\Gamma^T)/2 \otimes (P^{-\frac{1}{2}}AP^{-\frac{1}{2}}).
\]
Then,
\begin{align*}
   & \lambda_{\min}\left (I_s\otimes (P^{-\frac{1}{2}}MP^{-\frac{1}{2}}) +\Dt{}\, (\Gamma+\Gamma^T)/2
   \otimes (P^{-\frac{1}{2}}AP^{-\frac{1}{2}}) \right )
   \\
   &\ge \lambda_{\min}\left (I_s\otimes (P^{-\frac{1}{2}}MP^{-\frac{1}{2}}) \right )
      +\Dt{}\, \lambda_{\min}\left ( (\Gamma+\Gamma^T)/2
      \otimes (P^{-\frac{1}{2}}AP^{-\frac{1}{2}}) \right )
      .
\end{align*}
In case of $\lambda_{\min}((\Gamma + \Gamma^T)/2) \ge 0$ or if $\lambda_{\min}((\Gamma + \Gamma^T)/2) < 0$ and $\Dt$ is chosen such that \cref{dt-0} holds, then the symmetric part of $(I_s\otimes P^{-\frac{1}{2}}) (I_s\otimes M+\Dt{}\, \Gamma \otimes A) (I_s\otimes P^{-\frac{1}{2}})$ is positive definite and \cref{eq:irk:1:0,eq:irk:1:1} follow from the above results and \cref{cond-0,eq:kronLmin}.
\end{proof}

\begin{table}[t]\centering{}%
\caption{%
   Eigenvalues and singular values of the implicit RK coefficient matrix $\Gamma$ and the minimum eigenvalue of
   $(\Gamma+\Gamma^T)/2$.}\label{tab:t1}%
\begin{tabular}{lcllll}
   \toprule
   Method         & Order & Eigenvalues     & $\sigma_{\max}$ & $\sigma_{\min}$ &
      $\lambda_{\min}((\Gamma+\Gamma^T)/2)$\\
   \midrule
   Gauss          & 4   & $0.25\pm 0.1443 i$             & 0.6319 & 0.1319 &\phantom{-}0\\
   Gauss          & 6   & 0.2153, $0.1423 \pm 0.1358 i$  & 0.6629 & 0.0635 &-0.0563\\
   Radau IA       & 3   & $0.3333\pm 0.2357 i$           & 0.5    & 0.3333 &\phantom{-}0.25\\
   Radau IIA      & 5   & 0.2749, $0.1626\pm 0.1849 i$   & 0.8023 & 0.0923 &-0.0822\\
   Lobatto IIIA   & 4   & 0, $0.25\pm 0.1443 i$          & 0.7947 & 0      &-0.0736\\
   \bottomrule
\end{tabular}%
\end{table}

\begin{remark}\label{rem:irk:1}
In case of $P = I_{N_{vi}}$, $M$, $M_D$, $M_{lump}$, and $M_D + \Dt{}\, A_D$, the estimates for the eigenvalues $\lambda_{\min}(P^{-\frac{1}{2}} M P^{-\frac{1}{2}})$, $\lambda_{\max}(P^{-\frac{1}{2}} M P^{-\frac{1}{2}})$, $\lambda_{\min}(P^{-\frac{1}{2}} A P^{-\frac{1}{2}})$, and $\lambda_{\max}(P^{-\frac{1}{2}} A P^{-\frac{1}{2}})$ are given by \cref{lem:M-maxmin,lem:A-maxmin}.
The quantities $\sigma_{\max}(\Gamma)$ and $\lambda_{\min}((\Gamma+\Gamma^T)/2)$ are given in \cref{tab:t1} for five implicit Runge--Kutta methods.
One can see that $\sigma_{\max}(\Gamma) = \mathcal{O}(1)$ but $\lambda_{\min}((\Gamma+\Gamma^T)/2)$ can be zero, positive, and negative.

For the 4th-order Gauss and 3rd-order Radau IA methods, $\lambda_{\min}((\Gamma+\Gamma^T)/2)\ge 0$ and \cref{dt-0} poses no constraint on the choice of $\Dt$.
\Cref{eq:irk:1:0} indicates that the conditioning of the system \cref{eq:rk:3} resulting from implicit RK integration behaves more or less like that of $M + \Dt{}\, A$ for the implicit Euler method, and the bound also involves three mesh-dependent factors and one of them can be eliminated effectively by diagonal preconditioning (see \cref{rem:scaling}).

For the other three methods in \cref{tab:t1} with $\lambda_{\min}((\Gamma+\Gamma^T)/2) < 0$, \cref{dt-0} gives an upper bound on possible $\Dt$.
Note that this condition can be serious and lead to a condition like $\Delta t = \mathcal{O}(h^2)$.
\end{remark}

\subsection{Diagonally implicit RK (DIRK) methods}\label{sec:dirk}

The coefficient matrix $\Gamma = {(\gamma_{kj})}_{k,j=1}^s$ of a DIRK method is a lower triangular matrix and the system \cref{eq:rk:1} is solved by successively solving $s$ linear systems with $ M + \Dt{}\, \gamma_{jj} A$, $j = 1, \dotsc, s$.
The conditioning of $ M + \Dt{}\, \gamma_{jj} A$ is therefore similar to that of $M + \Dt{}\, A$ and, hence, the analysis in \cref{sec:euler} also applies to DIRK methods.

\subsection{General implicit RK methods: successive solution}\label{sec:irk-2}

A successive solution procedure transforms the large system \cref{eq:rk:1} into a number of smaller systems, which are then solved successively.
We adopt the approach of Butcher~\cite{But76} and Bickart~\cite{Bic77} and carry out the transformation using the Jordan normal form of the RK matrix.
To keep our analysis applicable to methods with a singular $\Gamma$, we use the Jordan normal form of $\Gamma$ instead of $\Gamma^{-1}$, which is the conventional choice (e.g.,\ see~\cite[p.~131]{HaiWan96}).

Let the Jordan normal form of $\Gamma$ be
\begin{equation}
\Gamma = T \begin{bmatrix} J_1 & & \\ & \ddots & \\ & & J_p \end{bmatrix} T^{-1},
\label{Jordan-1}
\end{equation}
where $T$ is a real invertible matrix and $J_j$, $j = 1, \dotsc, p$, are the Jordan blocks, which either have the form
\begin{equation}
   J_j = \begin{bmatrix} \mu_j & 1 & & \\ & \mu_j & \ddots & \\ & & \ddots & 1 \\
   & & & \mu_j \end{bmatrix}
   ,
   \quad \mu_j \in \R,
   \label{eq:Jordan:2}
\end{equation}
or
\begin{equation}
   J_j = \begin{bmatrix} C_j & I_2 & & \\ & C_j & \ddots & \\ & & \ddots & I_2 \\
   & & & C_j \end{bmatrix}
   \quad \text{with} \quad
   C_j = \begin{bmatrix} \alpha_j & \beta_j \\ - \beta_j & \alpha_j \end{bmatrix}
   ,
   \quad \alpha_j,\beta_j \in \R
   .
\label{eq:Jordan:3}
\end{equation}
Recall that the eigenvalues of $\Gamma$ are assumed to have nonnegative real parts, i.e., $\mu_j, \alpha_j \ge 0$.

\begin{theorem}\label{thm:irk:2}
Assume that the eigenvalues of the coefficient matrix $\Gamma$ for a given implicit RK method have nonnegative real parts and the system \cref{eq:rk:1} is solved by using the Jordan normal form of $\Gamma$ to transform it into smaller systems.
Then, the conditioning of implicit RK integration of \cref{eq:fem} is determined by the conditioning of
\begin{equation}
   M + \mu_j \Dt{}\, A
   \quad \text{and} \quad
   I_2\otimes M + \Dt{}\, C_j \otimes A,
   \label{eq:irk-2-1}
\end{equation}
where $\mu_j$ and $\alpha_j \pm i \beta_j$ are the real and complex eigenvalues of $\Gamma$, respectively.

Moreover, for any symmetric and positive definite preconditioner $P$,
\begin{multline}
   \tilde{\kappa} \left( (I_2\otimes P^{-\frac{1}{2}}) (I_2\otimes M+\Dt{}\, C_j \otimes A) (I_2\otimes P^{-\frac{1}{2}})\right)
   \\
   \le \frac{\lambda_{\max}(P^{-\frac{1}{2}} M P^{-\frac{1}{2}})
      + \Dt{}\, \sqrt{\alpha_j^2 + \beta_j^2} \, \lambda_{\max}(P^{-\frac{1}{2}} A P^{-\frac{1}{2}})}
   {\lambda_{\min}(P^{-\frac{1}{2}} M P^{-\frac{1}{2}})
              + \Dt\, \alpha_j \, \lambda_{\min}(P^{-\frac{1}{2}} A P^{-\frac{1}{2}})}
   .
   \label{eq:irk-2-2}
\end{multline}
\end{theorem}

\begin{proof}
Using \cref{Jordan-1}, we can rewrite \cref{eq:rk:3} as
\begin{align*}
   \mathcal{A} &= \left(T\otimes I_{N_{vi}}\right)
   \left( I_s\otimes M + \Dt{}\, \begin{bmatrix} J_1 & & \\ & \ddots & \\ & & J_p \end{bmatrix} \otimes A \right)
   \left(T^{-1} \otimes I_{N_{vi}}\right)
   \\
   &= \left(T\otimes I_{N_{vi}}\right)
   \begin{bmatrix} I_{n_1}\otimes M + \Dt{}\, J_1\otimes A & & \\ & \ddots & \\ & &
      I_{n_p}\otimes M + \Dt{}\, J_p \otimes A \end{bmatrix}
      \left(T^{-1} \otimes I_{N_{vi}}\right),
\end{align*}
where $n_j$ is the size of $J_j$ (with $\sum_{j=1}^p n_j = s$).
Since the systems with the coefficient matrix $T\otimes I_{N_{vi}}$ or $T^{-1} \otimes I_{N_{vi}}$ can be solved directly and efficiently, we only need to consider the iterative solution of the systems associated with
\begin{equation}
   I_{n_j}\otimes M + \Dt{}\, J_j\otimes A, \quad j = 1, \dotsc, p.
   \label{rk-4}
\end{equation}

For a Jordan block in the form \cref{eq:Jordan:2}, the above matrix becomes
\[
I_{n_j}\otimes M + \Dt{}\, J_j\otimes A
= \begin{bmatrix} M + \mu_j \Dt{}\, A & \Dt{}\, A & & \\ & M + \mu_j \Dt{}\, A & \ddots & \\
& & \ddots & \Dt{}\, A \\
& & & M + \mu_j \Dt{}\, A \end{bmatrix} ,
\]
which can be solved by successively solving $n_j$ systems with $ M + \mu_j \Dt{}\, A$ (backward substitution).

On the other hand, for a Jordan block in the form \cref{eq:Jordan:3}, the matrix \cref{rk-4} becomes
\begin{multline*}
I_{n_j}\otimes M + \Dt{}\, J_j \otimes A
   \\
   = \begin{bmatrix} I_2\otimes M + \Dt{}\, C_j \otimes A & \Dt{}\, I_2\otimes A & & \\
   & I_2\otimes M + \Dt{}\, C_j \otimes A & \ddots & \\
   & & \ddots & \Dt{}\, I_2\otimes A \\
   & & & I_2\otimes M + \Dt{}\, C_j \otimes A \end{bmatrix}
   ,
\end{multline*}
which, again, can be solved by successively solving $n_j/2$ systems with $I_2\otimes M + \Dt{}\, C_j \otimes A$.
Hence, the conditioning in this case is determined by the matrices in \cref{eq:irk-2-1}.

The analysis of \cref{sec:euler} can be used to estimate the condition number of $ M + \mu_j \Dt{}\, A$.
Estimate \cref{eq:irk-2-2} for the second matrix in \cref{eq:irk-2-1} is obtained similarly as for \cref{thm:irk:1},
except that we now have $\sigma_{\max}(C_j) = \sqrt{\alpha_j^2 + \beta_j^2}$ and $\lambda_{\min}((C_j+C_j^T)/2) = \alpha_j$.
\end{proof}

Note that \cref{rem:irk:1} applies to the current situation as well.
However, estimate \cref{eq:irk-2-2} can now be used for all methods in \cref{tab:t1} \emph{without any restriction on $\Dt$}.

\section{Numerical examples}\label{sec:numerics}
In the following examples we consider the IBVP~\cref{eq:IBVP} in 2d ($d = 2$) with $\Omega = {(0,1)}^2$ and homogeneous Dirichlet boundary conditions.
For the time integration we choose Radau~IA method of order~3 with fixed time steps
$\Dt{}\, = 10^{-1},\; 10^{-3},\; 10^{-5}$, and $\Dt{}\, = N^{-\frac{1}{d}} \sim h$, where $h$ is the average mesh size.
The \RKmethod{} coefficient matrix is given by
\[
   \Gamma = \frac{1}{12}
   \begin{pmatrix*}[r]
      3 & -3\\
      3 &  5
   \end{pmatrix*}
   ~
   \text{%
      with $\sigma_{\max}(\Gamma) = 0.5$
      and $\lambda_{\min}((\Gamma + \Gamma^T)/2) = 0.25$.%
   }
\]

\begin{example}\label{ex:radau5}
To compare the condition number $\tilde\kappa( I_s\otimes M + \Dt{}\, \Gamma \otimes A)$ of \RKmethod{} with the condition number $\kappa(M + \Dt{}\, A)$ of the implicit Euler method we consider two diffusion matrices: isotropic
\[
   \D = I \quad \text{(Laplace operator)}
\]
and anisotropic
\begin{equation}
   \D(x,y) =
   \begin{bmatrix}
      \cos \theta   & - \sin \theta\\
      \sin \theta   &   \cos \theta
   \end{bmatrix}
   \begin{bmatrix}
     10  & 0 \\
     0   & 0.1
   \end{bmatrix}
   \begin{bmatrix}
        \cos \theta  & \sin \theta\\
      - \sin \theta  & \cos \theta
   \end{bmatrix},
   \quad
   \theta = \pi \sin x \cos y
   ,
   \label{eq:anisotropic:D}
\end{equation}
with quasi-uniform (\cref{fig:meshes:uni}) and $\D^{-1}$-uniform (\cref{fig:meshes:duni}) meshes obtained using the mesh generator \textsl{bamg}~\cite{bamg}.
A quasi-uniform mesh can be also seen as $\D^{-1}$-uniform for $\D = I$.
\begin{figure}[t]\centering{}%
   \subcaptionbox{quasi-uniform\label{fig:meshes:uni}}
   [0.31\linewidth]{\centering{}%
      \includegraphics[width=0.31\linewidth, clip]{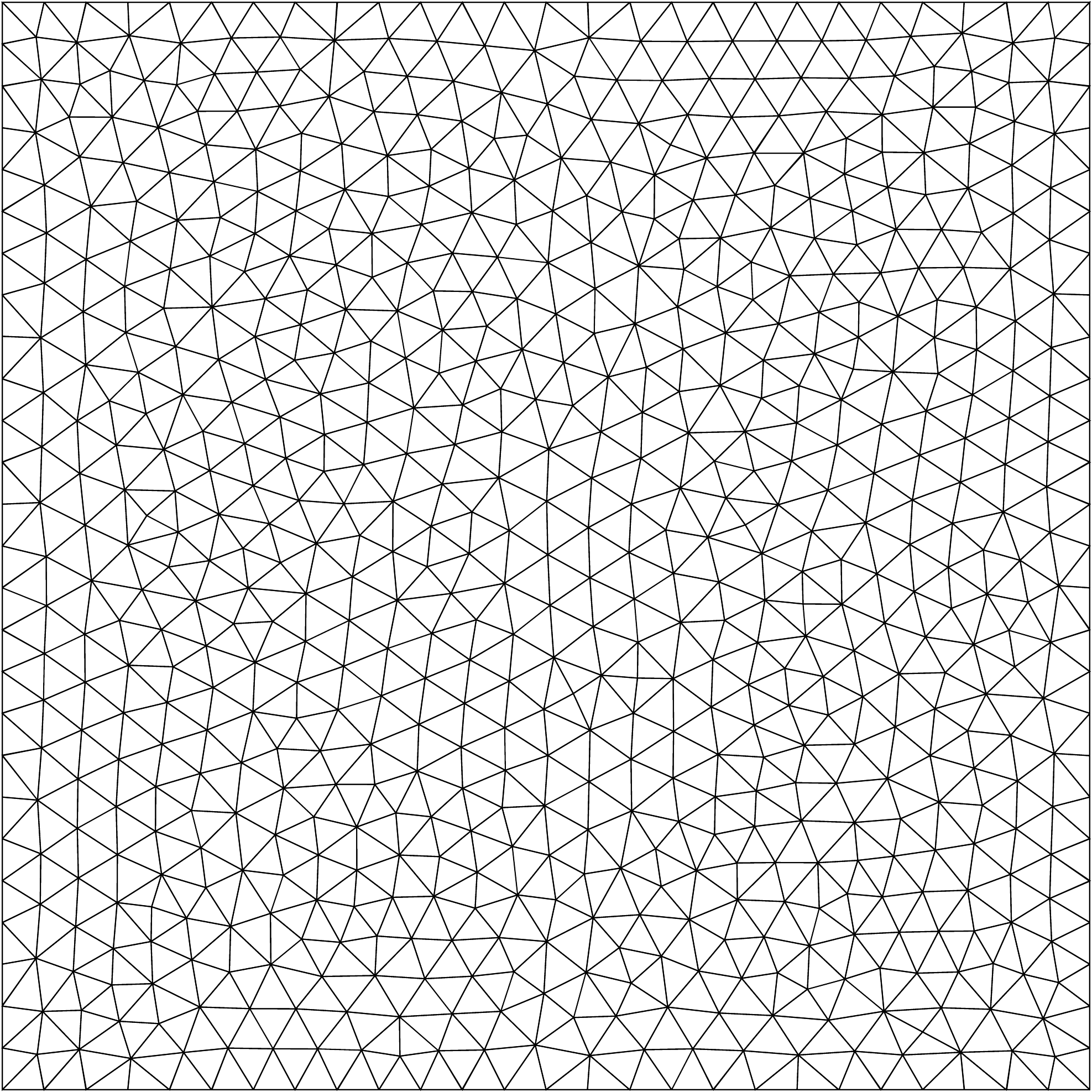}%
   }%
   \hfill{}%
   \subcaptionbox{$\D^{-1}$-uniform for \cref{eq:anisotropic:D}\label{fig:meshes:duni}}
      [0.31\linewidth]{\centering{}%
         \includegraphics[width=0.31\linewidth, clip]{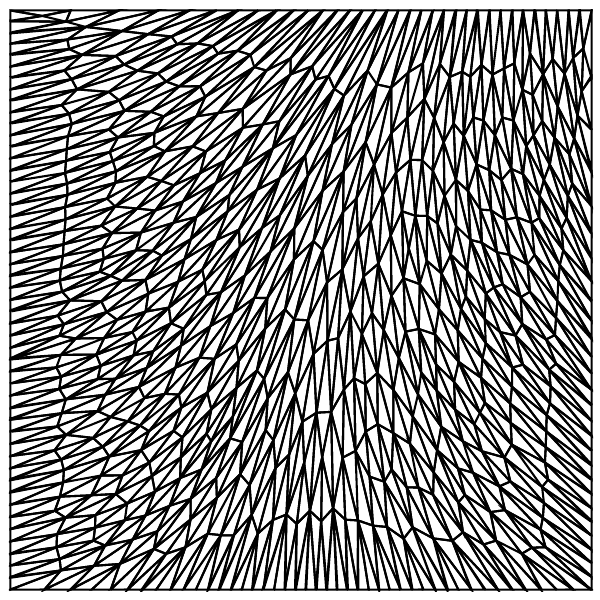}%
   }%
   \hfill{}%
   \subcaptionbox{adaptive for \cref{eq:tanh}\label{fig:meshes:adaptive}}
      [0.31\linewidth]{\centering{}%
         \includegraphics[width=0.31\linewidth, clip]{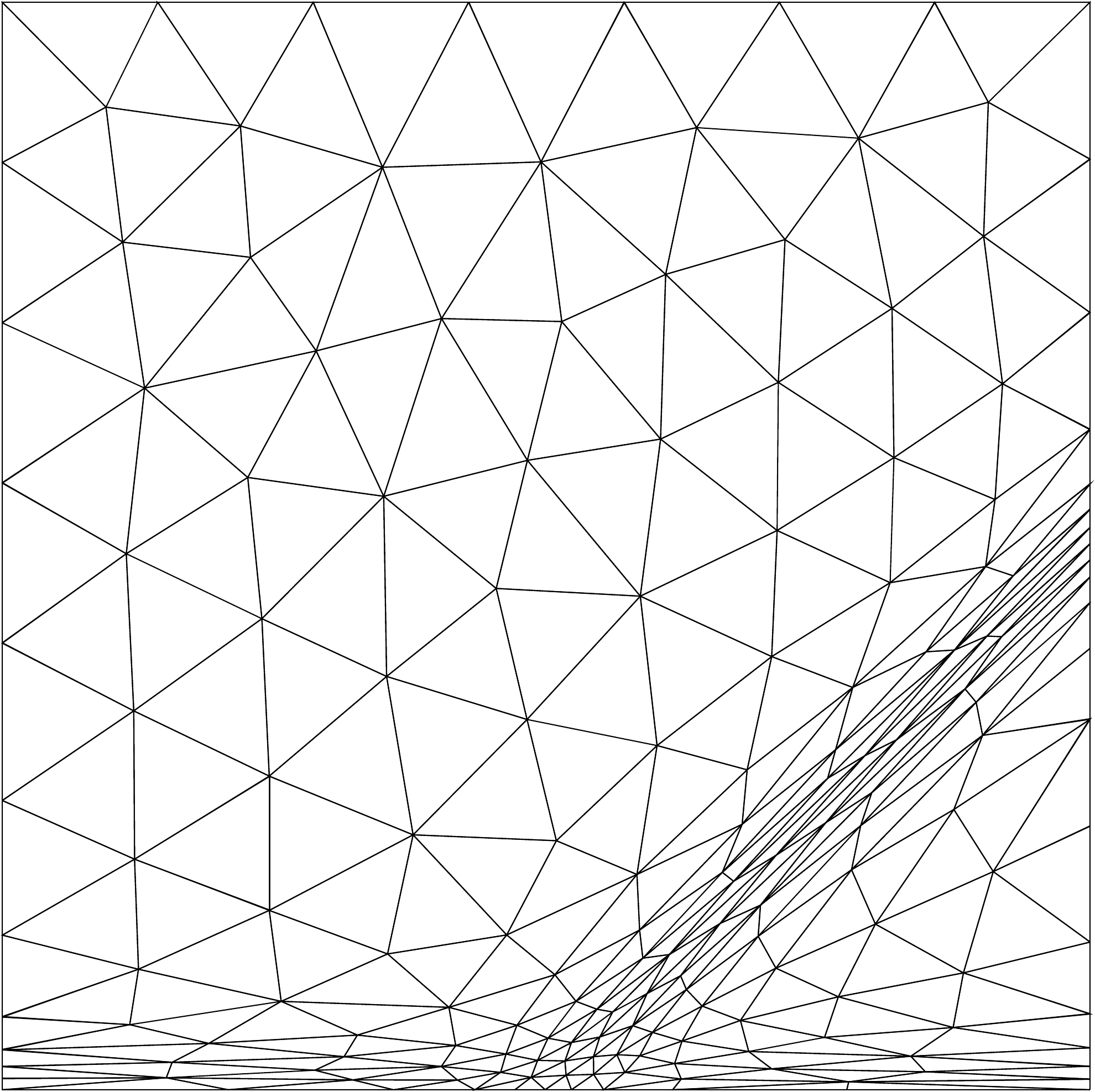}%
   }%
   \caption{Mesh examples.}\label{fig:meshes}%
\end{figure}

We expect the conditioning of \RKmethod{} to be similar to that of the implicit Euler (see \cref{eq:irk:1:1,rem:irk:1}).
This is verified in \cref{fig:kappa} (left column), which shows that the conditioning for the \RKmethod{} method has the same general behaviour as that for the implicit Euler scheme.
With increasing number of mesh elements, the order of conditioning is $\mathcal{O}(N^{\frac{2}{d}})$ for a fixed $\Dt{}$ and becomes $\mathcal{O}(N^{\frac{1}{d}})$ for $\Dt{}\, \sim h \sim N^{-\frac{1}{d}}$ for both methods, which is in perfect agreement with the theoretical prediction in \cref{rem:dt:h}.

Adapting towards the diffusion matrix improves the conditioning of the stiffness matrix~\cite{KamHuaXu14}.
This effect is observed in our test case as well: for anisotropic diffusion, the conditioning for diffusion-adapted meshes (\cref{fig:D10}) is noticeably smaller than that with quasi-uniform meshes (\cref{fig:D7}), especially as $N$ is getting larger.

To compare the exact values for the conditioning of \RKmethod{} with the bound \cref{eq:irk:1:1} in \cref{thm:irk:1}, we use the exact eigenvalues for the matrices $P^{-1/2}MP^{-1/2}$ and $P^{-1/2}AP^{-1/2}$.
\Cref{fig:kappa} (right column) shows that the bound \cref{eq:irk:1:0} is greater than the exact condition number but exhibits essentially the same behaviour as $N$ increases.

\begin{figure}[p]%
   \begin{subfigure}[t]{1.0\textwidth}%
      \plotRadauTEuler{0}{i-r1a3-ama1001-07}{4}{3}{ \RKshort}{ Euler}{o}{square}{title = \RKmethod{} vs.\ implicit Euler,}%
      \hfill{}%
      \plotRadauTEuler{0}{i-r1a3-ama1001-07}{4}{5}{ \RKshort}{ bound \cref{eq:irk:1:0}}{o}{*}{title = \RKmethod{} vs.\ bound \cref{eq:irk:1:0}, ytick pos=right,}%
      \caption{quasi-uniform grids (\cref{fig:meshes:uni})
         and the Laplace operator ($\D = I$)}\label{fig:I7}%
   \end{subfigure}%
   \\[\baselineskip]%
   \begin{subfigure}[t]{1.0\textwidth}%
      \plotRadauTEuler{0}{d-r1a3-ama1001-07}{4}{3}{ \RKshort}{ Euler}{o}{square}{}%
      \hfill{}%
      \plotRadauTEuler{0}{d-r1a3-ama1001-07}{4}{5}{ \RKshort}{ bound \cref{eq:irk:1:0}}{o}{*}{ytick pos=right,}%
      \caption{quasi-uniform grids (\cref{fig:meshes:uni})
         and the anisotropic $\D$ from \cref{eq:anisotropic:D}}\label{fig:D7}%
   \end{subfigure}%
   \\[\baselineskip]%
   \begin{subfigure}[t]{1.0\textwidth}%
      \plotRadauTEuler{0}{d-r1a3-ama1001-10}{4}{3}{ \RKshort}{ Euler}{o}{square}{}%
      \hfill{}%
      \plotRadauTEuler{0}{d-r1a3-ama1001-10}{4}{5}{ \RKshort}{ bound \cref{eq:irk:1:0}}{o}{*}{ytick pos=right,}%
      \caption{$\D^{-1}$-uniform grids (\cref{fig:meshes:duni})
         for the anisotropic $\D$ from \cref{eq:anisotropic:D}}\label{fig:D10}%
   \end{subfigure}%
   \caption{%
      Conditioning for \RKmethod{} (\RKshort) and implicit Euler (left)
      and a comparison of \RKmethod{} conditioning with the bound \cref{eq:irk:1:0} (right)
      as functions of $N$ (\cref{ex:radau5}).%
   }\label{fig:kappa}%
\end{figure}
\end{example}

\begin{example}\label{ex:scaling}
To illustrate the effect of the diagonal scaling, we consider $\D=I$ (Laplace operator), $P = M_D + \Dt{}\, A_D$,  $M_D$ and $M_{lump}$, and adaptive anisotropic meshes for the interpolation of the function
\begin{equation}
   u(x, y) = \tanh (60y) - \tanh (60 (x - y - 0.5)),
   \quad (x, y) \in {(0,1)}^2
   ,
   \label{eq:tanh}
\end{equation}
generated with \textsl{bamg}~\cite{bamg} using an anisotropic adaptive metric~\cite{Huang06}.
This function simulates the interaction between a boundary layer along the $x$-axis and a shock wave along the line $x = y + 0.5$ (see \cref{fig:meshes:adaptive} for a mesh example).

Consistently with \cref{rem:scaling}, \cref{fig:scaling:MA} shows that the diagonal scaling with $P = M_D + \Dt{}\, A_D$ (Euler) and $I_s\otimes M_D + \Dt{}\, \Gamma \otimes A_D$ (\RKmethod{}) reduces the effects caused by the mesh nonuniformity: the conditioning of both Euler and \RKmethod{} methods is reduced by a factor ranging from 3 ($\Dt = 10^{-1}$)
to 8 ($\Dt = 10^{-3},\; 10^{-5}$).

For the scaling with $P = M_D$, \cref{fig:scaling:MD} shows that the conditioning is getting worse for large $\Dt{}$ ($\Dt = 10^{-1}$) but is improving for small $\Dt{}$ ($\Dt = 10^{-3}$, $\Dt = 10^{-5}$).
To explain this, we recall that the diagonal entries of $M$ are proportional to the corresponding patch volumes (see \cref{lem:M:1}).
Hence, scaling with $\M_D$ improves the conditioning issues caused by the mesh \emph{volume-nonuniformity} (the same applies to $M$ and $M_{lump}$).
On the other hand, $\max_j A_{jj} \le \lambda_{\max}(A) \le (d+1) \max_j A_{jj}$~\cite[Lemma 4.1]{KamHuaXu14} and \cref{lem:A:2} imply that the largest eigenvalue of $A$ depends on the element shape (see also~\cite[section 4.1]{KamHuaXu14} and~\cite[section 3]{She02}).
Thus, $\kappa(M^{-1} A)$ is not necessarily smaller than $\kappa(A)$ since rescaling of $A$ with respect to the patch volumes does not necessarily improve the conditioning issues caused by the shape.
Therefore, in general, the overall effect depends on the magnitude of $\Dt{}$: $\kappa(I + \Dt{} M^{-1} A)$ is not necessarily better than $\kappa(M + \Dt{} A)$ for large $\Dt{}$, but we can expect that $\kappa(I + \Dt{} M^{-1} A) < \kappa(M + \Dt{} A)$ for small $\Dt{}$.

\begin{figure}[t]
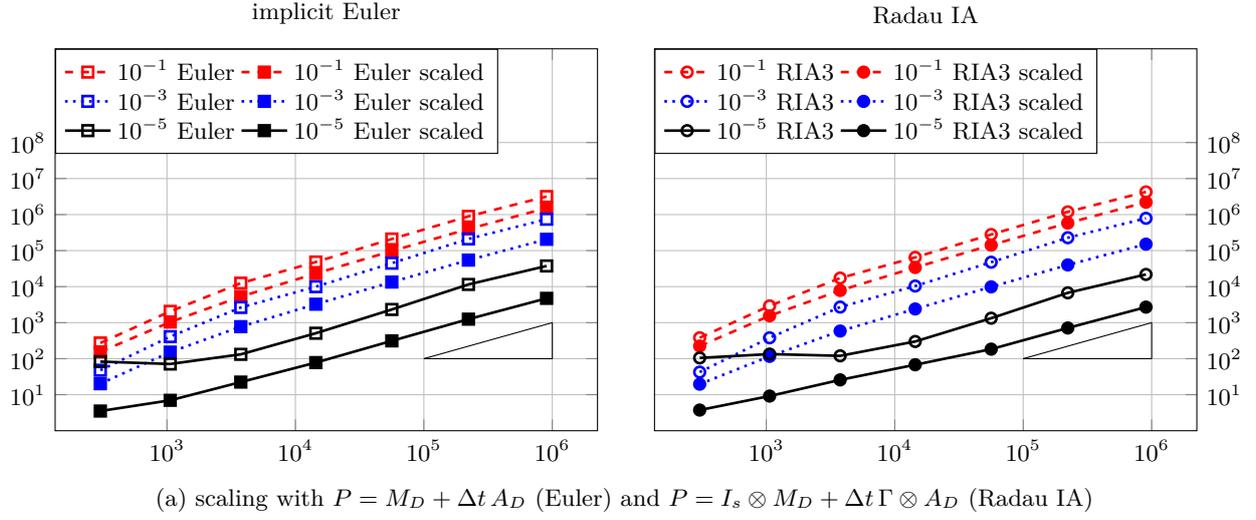
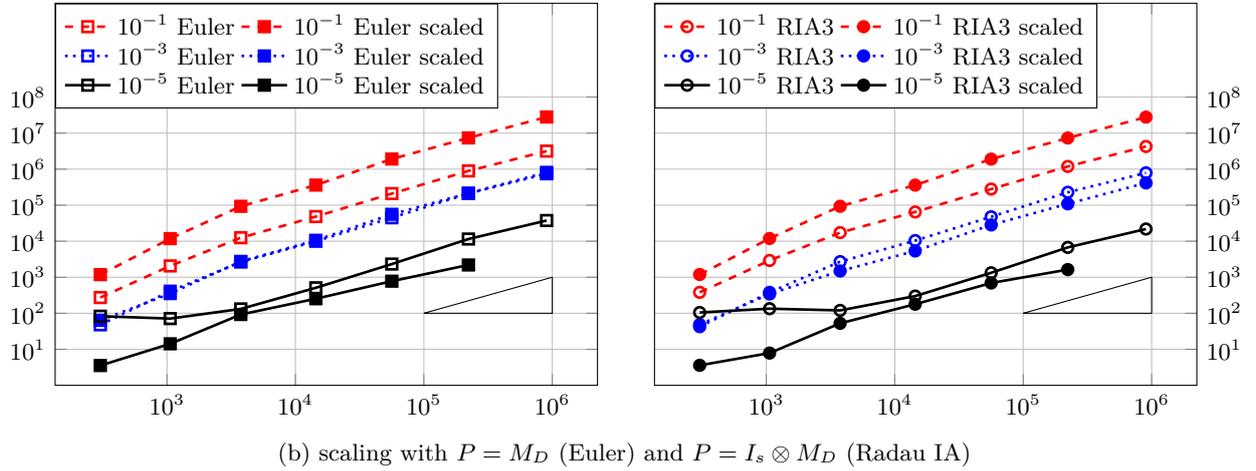

   \begin{subfigure}[t]{1.0\textwidth}%
      \plotScalingE{0}{1}{i-r1a3-tanh30}{3}{ Euler}{ Euler scaled}{square}{square*}{title = implicit Euler,}%
      \hfill{}%
      \plotScalingE{0}{1}{i-r1a3-tanh30}{4}{ \RKshort}{ \RKshort~scaled}{o}{*}{title = \RKmethod{}, ytick pos=right,}%
      \caption{scaling with $P = M_D + \Dt{}\, A_D$ (Euler)
         and $P = I_s\otimes M_D + \Dt{}\, \Gamma \otimes A_D$ (\RKmethod{})}\label{fig:scaling:MA}
   \end{subfigure}%
   \\[\baselineskip]%
   \begin{subfigure}[t]{1.0\textwidth}%
      \plotScalingE{0}{1MD}{i-r1a3-tanh30}{3}{ Euler}{ Euler scaled}{square}{square*}{}%
      \hfill{}%
      \plotScalingE{0}{1MD}{i-r1a3-tanh30}{4}{ \RKshort}{ \RKshort~scaled}{o}{*}{ytick pos=right,}%
      \caption{scaling with $P = M_D$ (Euler) and $P = I_s\otimes M_D$ (\RKmethod{})}\label{fig:scaling:MD}
   \end{subfigure}%
    \caption{%
      Conditioning of implicit Euler and \RKmethod{} (\RKshort) as function of $N$
      before and after a diagonal scaling
      for the Laplace operator and adaptive anisotropic meshes (\cref{ex:scaling}).%
   }\label{fig:scaling}%
\end{figure}
\end{example}

\section{Conclusions}\label{sec:conclusion}

For the FE equations, the conditioning of implicit RK methods with positive semidefinite symmetric part of the coefficient matrix  is comparable to the implicit Euler integration (\cref{thm:irk:1}).

If the smallest eigenvalue of the symmetric part of the coefficient matrix is negative,
\cref{dt-0} gives an upper bound on possible $\Dt$ so that the system matrix $I_s\otimes M + \Dt{}\, \Gamma \otimes A$ is positive definite.
This condition implies $\Dt \le C \frac{\lambda_{\min}(B)}{\lambda_{\max}(A)}$ and can lead to a condition $\Delta t = \mathcal{O}(h^2)$.%

For the successive solution procedure, the conditioning of the system matrix $I_s\otimes M + \Dt{}\, \Gamma \otimes A$ of an implicit RK integration is determined by two types of smaller matrices (\cref{thm:irk:2})
\[
   M + \mu_j \Dt{}\, A
   \quad \text{and} \quad
   I_2\otimes M + \Dt{}\, \begin{bmatrix} \alpha_j & \beta_j \\ - \beta_j & \alpha_j \end{bmatrix} \otimes A
   .
\]
The first matrix is similar to the implicit Euler method and corresponds to the real eigenvalues $\mu_j$ of the RK matrix  $\Gamma$ while the second corresponds to the complex eigenvalues $\alpha_j \pm i \beta_j$ of $\Gamma$.

There are three mesh-dependent factors that affect the conditioning of general implicit RK methods: the number of mesh elements (average mesh size), the mesh nonuniformity (in the Euclidean metric), and the mesh nonuniformity with respect to the inverse of the diffusion matrix.
Preconditioning by the diagonal part of the system matrix itself or the mass matrix or by the lumped mass matrix reduces the effects of the mesh nonuniformity in the Euclidean metric (\cref{thm:euler:2,thm:euler:3}).
These results are consistent with previous studies for the boundary value problems~\cite{AinMcLTra99, BanSco89, GraMcL06, KamHuaXu14} and explicit integration of linear diffusion problems~\cite{HuaKamLan15,HuaKamLan16}.

Similar result can be shown for higher order FE approximations.
The major result is based on the preliminary estimates for the mass and stiffness matrices (\cref{lem:M:1,lem:M:2,lem:A:1,lem:A:1:1,lem:A:3}) which, in their essential part, hold for higher order as well (with different constants depending of the order).
For the FE, a symmetric positive definite (possibly lumped) mass matrix is equivalent to a diagonal matrix of node patch volumes~\cite[Lemmas~2.4]{HuaKamLan15} so that estimates equivalent to \cref{lem:M:1,lem:M:2} hold for higher order as well (with different constants).
For the stiffness matrix, estimates equivalent to \cref{lem:A:1,lem:A:1:1} are given in~\cite[Lemmas~2.2 and 2.3]{HuaKamLan15} and the derivation of the bound on $\lambda_{\min}(A)$ in \cref{lem:A:3} does not depend on the FE order (up to a constant) and therefore holds also for higher order FE.


\end{document}